\documentclass[12pt]{article}
\usepackage{amssymb}
\usepackage{amsmath}
\usepackage{graphics}

\usepackage{url}

\usepackage{hyperref}
\usepackage{fancyhdr}
\usepackage[authoryear, round]{natbib}

\DeclareRobustCommand{\VAN}[3]{#2}

\begin{document}

\newenvironment {proof}{{\noindent\bf Proof.}}{\hfill $\Box$ \medskip}

\newtheorem{thm}{Theorem}[section]
\newtheorem{propn}[thm]{Proposition}
\newtheorem{lemma}[thm]{Lemma}
\newtheorem{eg}[thm]{Example}
\newtheorem{defn}[thm]{Definition}
\newtheorem{remark}[thm]{Remark}
\newtheorem{notn}[thm]{Notation}
\newtheorem{corollary}[thm]{Corollary}
\newtheorem{conjecture}[thm]{Conjecture}
\newtheorem{assumption}[thm]{Assumption}
\newtheorem{condn}[thm]{Condition}
\newtheorem{example}[thm]{Example}

\newcommand\numberthis{\addtocounter{equation}{1}\tag{\theequation}}

\newcommand{\IP}{\mathbb P}
\newcommand{\IQ}{\mathbb Q}
\newcommand{\IE}{\mathbb E}
\newcommand{\IR}{\mathbb R}
\newcommand{\IZ}{\mathbb Z}
\newcommand{\IN}{\mathbb N}
\newcommand{\IT}{\mathbb T}
\newcommand{\IC}{\mathbb C}

\newcommand{\rmP}{\mathrm{P}}
\newcommand{\rmE}{\mathrm{E}}
\newcommand{\rmT}{\mathrm{T}}
\newcommand{\rmt}{\mathrm{t}}

\newcommand{\tX}{\tilde{X}}

\newcommand{\hX}{\hat{X}}
\newcommand{\hS}{\hat{S}}
\newcommand{\hq}{\hat{q}}
\newcommand{\hQ}{\hat{Q}}

\newcommand{\cP}{\mathcal{P}}
\newcommand{\cO}{\mathcal{O}}
\newcommand{\cA}{\mathcal{A}}
\newcommand{\cL}{\mathcal{L}}
\newcommand{\cG}{\mathcal{G}}
\newcommand{\cC}{\mathcal{C}}
\newcommand{\cV}{\mathcal{V}}
\newcommand{\cE}{\mathcal{E}}
\newcommand{\cT}{\mathcal{T}}
\newcommand{\cD}{\mathcal{D}}
\newcommand{\cZ}{\mathcal{Z}}
\newcommand{\cB}{\mathcal{B}}
\newcommand{\cK}{\mathcal{K}}
\newcommand{\cN}{\mathcal{N}}
\newcommand{\cX}{\mathcal{X}}
\newcommand{\cF}{\mathcal{F}}
\newcommand{\dd}{\; \mathrm{d}}
\newcommand{\ind}{\mathbf{1}}
\newcommand{\ldual}{\langle}
\newcommand{\rdual}{\rangle}
\newcommand{\lquad}{\langle}
\newcommand{\rquad}{\rangle}
\newcommand{\ballg}[1]{B_{#1}}
\newcommand{\ball}[2]{B_{#1}(#2)}
\newcommand{\stopped}[1]{\widehat{#1}}
\newcommand{\conditional}{|}
\newcommand{\compliment}[1]{#1^c}
\newcommand{\cinftyc}{C^\infty_c}
\newcommand{\abs}{|}
\newcommand{\limit}[2]{\lim_{#1 \to #2}}
\newcommand{\Id}{\text{Id}}
\pagestyle{fancy}
\lhead{}
\chead{SuperBrownian motion and the spatial Lambda-Fleming-Viot process}
\rhead{}

\numberwithin{equation}{section}

\renewcommand {\theequation}{\arabic{section}.\arabic{equation}}
\def \non{{\nonumber}}
\def \hat{\widehat}
\def \tilde{\widetilde}
\def \bar{\overline}

\title{\large{\bf SuperBrownian motion and the spatial Lambda-Fleming-Viot
process}}
                                           
\author{ \begin{small}
\begin{tabular}{ll}                              
Jonathan A. Chetwynd-Diggle \thanks{Supported by EPSRC grant
number EP/L015811/1}
&
Alison M. Etheridge 
\\   
Department of Mathematics
&
Department of Statistics\\       
Oxford University & Oxford University\\                   
Radcliffe Observatory Quarter, Woodstock Rd & 24-29 St Giles\\                                                         
Oxford OX2 6GG & Oxford OX1 3LB\\
UK & UK \\                        
jonathan.chetwynd-diggle@maths.ox.ac.uk & etheridg@stats.ox.ac.uk     \\
\end{tabular}
\end{small}}

\maketitle

\begin{abstract}
It is well known that the 
dynamics of a subpopulation of individuals of a rare type 
in a Wright-Fisher diffusion can be approximated by a
Feller branching process. Here we establish
an analogue of that result for a spatially distributed population
whose dynamics are described by a spatial Lambda-Fleming-Viot process
(SLFV).
The subpopulation of rare individuals is then approximated by 
a superBrownian motion.
This result mirrors 
\cite{cox/durrett/perkins:2000}, where it is shown 
that when suitably rescaled, sparse 
voter models converge to superBrownian motion. 
We also prove the somewhat more surprising result, that by choosing the 
dynamics of the SLFV appropriately we can recover superBrownian motion 
with stable branching in an analogous way. This is a spatial analogue of 
(a special case of) results of \cite{bertoin/legall:2006}, who show 
that the generalised Fleming-Viot process that is dual to the beta-coalescent, 
when suitably rescaled, converges to a continuous state branching 
process with stable branching mechanism.

\vspace{.1in}


\vspace{.1in}

\noindent {\bf MSC 20}10 {\bf Subject Classification:}  Primary:
60F05,
60G57,
60J25,
60J68,
60J80
\\
Secondary:   
60G51,
60G55,
60J75,
92D10

\end{abstract}

\tableofcontents

\setcounter{equation}{0}

\section{Background}
\label{background}

Our aim in this paper is to establish a relationship between two, at first sight, 
very different classes of measure-valued processes. The first, the spatial
Lambda-Fleming-Viot processes, is a collection of models for
the evolution of frequencies of different 
genetic types in a population that is dispersed across a spatial continuum. The second
is the (finite and infinite variance) superBrownian motions. 
Our motivation is two-fold. On the one hand, we add to the panoply of processes that
converge to superBrownian motion; on the other, we address a question of some interest
in population genetics: how does the frequency of a rare neutral 
mutation evolve in a spatially distributed population?

SuperBrownian motion, or the Dawson-Watanabe superprocess,
was introduced independently by \cite{watanabe:1968} and \cite{dawson:1975} as a continuous time and space
approximation to systems of branching Brownian motions.  
In this way it can be thought of as a spatial analogue of the 
Feller diffusion approximation
to critical (or near-critical) Galton-Watson branching processes.
We shall recall its definition in 
Section~\ref{superbrownian motion subsection} below.

In addition to the huge literature exploring the rich mathematical structure of
superBrownian motion, over the last two decades an increasing body of evidence has
emerged that it is a universal scaling limit of critical 
interacting particle systems above a critical dimension. 
It has been obtained as a limit of lattice trees (above 8 dimensions, e.g.~\cite{holmes:2008}), 
oriented 
percolation (above 4 dimensions, \cite{vanderhofstad/slade:2003}), 
the contact process (above 4 dimensions, e.g.~\cite{vanderhofstad/sakai:2010}),
the
voter model (in two or more dimensions, e.g.~\cite{cox/durrett/perkins:2000})
and the Lotka-Volterra model (\cite{cox/perkins:2005}).
By changing the range of the interaction with the scaling, one can also obtain
it from the contact process in lower dimensions (\cite{cox/durrett/perkins:1999}).
These analyses prove convergence of finite-dimensional distributions; \cite{vanderhofstad/holmes/perkins:2017} provide a tightness criterion that 
allows the extension to convergence on path space and apply it
to the example of sufficiently spread out lattice trees above 8 dimensions. 
We also refer to that paper for a more complete list of references.

For populations that are not spatially distributed, one classically models frequencies 
of different genetic types (usually refered to as alleles) through a Wright-Fisher 
or a Cannings model. Suppose that we are interested in the proportion of 
individuals of a particular type, that
we shall call type $1$. Under the Wright-Fisher
model, when suitably scaled, this proportion
converges to the Wright-Fisher diffusion. 
If type $1$ is rare, the absolute number of type $1$ individuals evolves approximately 
according to a branching process which, under the same scaling,
converges to a Feller diffusion. This branching process
approximation for the rare type
has been used extensively in the population genetics literature
and so it is natural to try to establish analogous results for spatially
distributed populations.

In one spatial dimension, the Wright-Fisher diffusion has a stochastic pde counterpart:
\begin{equation}
\label{WF spde}
dw_t(x)=\frac{1}{2}\Delta w_t(x)\dd t +\sqrt{\frac{1}{K} w_t(x)\big(1-w_t(x)\big)}
W(\dd t,\dd x),
\end{equation}
where $w_t(x)$ denotes the proportion of the population at spatial position $x$ at
time $t$ that is of type $1$, $K$ is the local population density,
and $W(dt,dx)$ is a space time white noise. Formally at 
least, if type $1$ is rare, this reduces to
$$
dw_t(x)=\frac{1}{2}\Delta w_t(x)\dd t+\sqrt{\frac{1}{K} w_t(x)}W(\dd t,\dd x),
$$
and if we set $X_t=K w_t$ to recover absolute numbers rather than proportions, 
the type $1$ population 
is modelled by 
$$
dX_t(x)=\frac{1}{2}\Delta X_t(x)\dd t +\sqrt{X_t(x)}W(\dd t,\dd x),
$$
which is the stochastic pde governing the density with respect to Lebesgue measure of
the (finite variance) superBrownian motion, and so it is certainly reasonable to hope to describe 
establishment of rare alleles in one dimensional populations using superBrownian motion.

In dimensions two and higher, equation~(\ref{WF spde}) has no solution and so we need
an alternative approach to modelling allele frequencies in higher dimensional spatial
continua. The obstructions to finding such an approach, often refered to as `the
pain in the torus', are well documented. We refer to \cite{barton/etheridge/veber:2013} for a survey.
The spatial Lambda-Fleming-Viot process (SLFV) introduced in \cite{etheridge:2008}, overcomes
the pain in the torus to provide a class of models for allele frequencies in populations
distributed across spatial continua of any dimension. 
The first rigorous construction is in \cite{barton/etheridge/veber:2010}.
The SLFV can be thought of as the spatial
counterpart of the `generalised Fleming-Viot process' (that
we shall refer to as the Lambda-Fleming-Viot process in what follows) 
of \cite{bertoin/legall:2003} and,
just as for their model, comes with a consistent
`backwards in time' model (a spatial analogue of the Lambda-coalescent)
for the genealogies describing relatedness between 
genes in individuals sampled from the population. 
We recall the definition of the process in Section~\ref{slfv subsection}.

As a special case of the results in \cite{etheridge/veber/yu:2014},
in one spatial dimension one can recover~(\ref{WF spde})
as a scaling limit of a particular SLFV. 
The corresponding scaling in higher dimensions leads to
the (deterministic) heat equation.
This can perhaps best be understood as a `law of large numbers' effect.
In particular, the initial conditions taken in that paper don't correspond to `rare' alleles.
As we shall see, we can recover superBrownian motion from the SLFV in arbitrary spatial
dimensions, but only if we take a sufficiently `sparse' initial condition.
This should of course be compared to the results of \cite{cox/durrett/perkins:2000}, who recover superBrownian motion from sparse
voter models and our analysis in the finite variance case
owes a great deal to that paper.
We should also mention the work of \cite{freeman:2010},
in which he introduces a very close relative of the SLFV, 
which he calls a `bursting process',
on $\IZ^d$ and shows that for $d\geq 3$, started from sparse initial conditions
and suitably scaled, that process too converges to a
superBrownian limit.

In the discussion up to this point
we have (implicitly) considered the finite variance superBrownian motion.
Where our work diverges from the body of work described above is that we are 
also able
to obtain superBrownian motions with stable branching mechanisms 
from particular choices of the SLFV. 
Such superprocesses are the spatial analogue of the continuous state branching processes
sometimes known as stable branching processes. In \cite{birkner/blath/capaldo/etal:2005},
it is shown that the special class of Lambda-Fleming-Viot processes
that are dual to the so-called Beta-coalescents
can be obtained as time-changed stable branching processes, revealing a deep 
connection between the two classes of processes. \cite{bertoin/legall:2006} show that 
in much the same way
as the Feller diffusion describes evolution of a rare allele in a population
evolving according to the Wright-Fisher diffusion, stable branching describes
the evolution of a rare allele under this Lambda-Fleming-Viot process
(see \cite{lambert/schertzer:2016} for a `backwards in time' analogue).
We provide the `back of the envelope' calculation that explains Bertoin and
Le~Gall's result in Section~\ref{heuristics}.
What is more 
surprising is that we can extract a
superBrownian motion with stable branching mechanism
from a sequence of rescaled SLFVs.
First, the conditions on the `Lambda'-measure under which we can 
construct the SLFV are more restrictive than those
under which we can construct the (non-spatial) Lambda-Fleming-Viot process.
Second, the spatial motion of individuals in the SLFV is intricately connected to the
reproduction mechanism, yet we are 
trying to produce a limit in which spatial motion is continuous and reproduction is 
discontinuous.  On the other hand, in \cite{etheridge/veber/yu:2014} the analogue of~(\ref{WF spde}) with the Laplacian replaced by the 
generator of a symmetric stable process is obtained as a scaling limit 
of an SLFV. In that case, in the limit the spatial motion is discontinuous and
the reproduction mechanism continuous.

The rest of this paper is laid out as follows. In 
Section~\ref{defns and results} we
remind the reader of the definitions of the superBrownian motion and the SLFV 
before stating our main results. We also give a heuristic explanation 
of our results.
In Section~\ref{martingale problem prelimit} we provide 
martingale characterisations of the scaled SLFVs, from which,
in Section~\ref{limits}, we
formally identify the limiting objects,
deferring tightness to Section~\ref{tightness} and the 
proof of some 
key estimates to Section~\ref{indeplinsection}.
The proof of convergence 
follows in
Section~\ref{convergence section}.

\section{Definitions and statement of results}
\label{defns and results}

Before stating our results in Section~\ref{main results}
below, we fix notation and define our
two classes of processes.

\subsection{SuperBrownian motion}
\label{superbrownian motion subsection}

We shall characterise superBrownian motion through a martingale problem. It is 
convenient to use distinct formulations in the finite and infinite variance
cases. 
For an introduction to superprocesses and, in particular, their construction as 
scaling limits of branching particle systems, we refer to \cite{dawson:1993}, \cite{perkins:2002} and \cite{etheridge:2000}.

A complete filtered probability space $(\Omega, \cF,\cF_t,\IP)$ will be 
implicit throughout.
We write ${\mathcal M}_F(\IR^d)$ for the space of finite measures on $\IR^d$, equipped with
the topology of weak convergence, and $C_0^k(\IR^d)$ for the space of 
$k$ times 
differentiable functions $\phi:\IR^d\rightarrow\IR$, vanishing at
infinity, and such that $\phi$ and 
its derivatives
up to $k$th order are bounded with norm
$$\|\phi\|_{C^k}=\max_{0\leq l\leq k}\left\|
\frac{\partial^l\phi}{\partial x^l}\right\|_{\infty}.$$

\begin{defn}[Finite variance superBrownian motion] \label{superBMdefn}
The {\em finite variance superBrownian motion} is the unique 
${\mathcal M}_F(\IR^d)$-valued Markov process $\{X_t\}_{t\geq 0}$ 
with continuous sample paths such that for each non-negative
$\phi\in C_0^3(\IR^d)$, the process
\begin{align*}
M_t(\phi) := \ldual X_t , \phi \rdual - \ldual X_0, \phi \rdual 
- \int_0^t \ldual X_s, \frac{m}{2}\Delta \phi \rdual \dd s
,
\end{align*}
is a continuous, square integrable martingale with quadratic variation given by
\begin{align*}
\lquad M(\phi)\rquad_t = 2\kappa \int_0^t \ldual X_s, \phi^2 \rdual \dd s,
\end{align*}
where $m,\kappa >0$ are constants.
\end{defn}

We shall not consider the most general possible superBrownian motions. Instead, we
restrict ourselves to those that arise as scaling limits of branching
Brownian motions
with offspring distribution in the domain of attraction of a stable law. These are 
naturally parametrised by a parameter $\beta\in (0,1)$ (with $\beta=1$ corresponding
to the finite variance case).

\begin{defn}[SuperBrownian motion with stable branching law]
\label{infvarDWdefn}
The {\em superBrownian motion with stable branching law of 
parameter $\beta \in (0,1)$}
is the unique ${\mathcal M}_F(\IR^d)$-valued Markov process $\{X_t\}_{t\geq 0}$
with c\`adl\`ag sample paths such that for 
each non-negative $\phi\in C_0^3(\IR^d)$, the process 
\begin{align}
\label{beta branching martingale}
M_t(\phi) := exp(-\ldual X_t, \phi \rdual)
- exp(-\ldual X_0, \phi \rdual)
- \int_0^t \ldual X_s, -\frac{m}{2}\Delta \phi + \kappa \phi^{1 + \beta} \rdual 
\exp(- \ldual X_s , \phi \rdual )\dd s
,
\end{align}
is a martingale, where $m,\kappa > 0 $ are constants.
\end{defn}

\subsection{The Spatial Lambda-Fleming-Viot Process}
\label{slfv subsection}

We now introduce the SLFV processes. 
In fact there is a much richer class of these processes than those we consider
here, incorporating, for example, various forms of natural selection. For a 
(somewhat out of date) survey we refer to \cite{barton/etheridge/veber:2013}.
We restrict 
ourselves to a population in which there are just two genetic types which we label
by $\{0,1\}$.
At each time $t$, the 
random function $\{w_t(x),\, x\in \IR^d\}$ 
is defined, up to a Lebesgue null set of $\IR^d$, by
\begin{equation}
\nonumber
w_t(x):= \hbox{ proportion of type }1\hbox{ at spatial position }x\hbox{ at time }t.
\end{equation}
A construction of an appropriate state space for $x\mapsto w_t(x)$ can be found in \cite{veber/wakolbinger:2015}.
Using the identification
$$
\int_{\IR^d\times \{0,1\}} f(x,\alpha) M(\dd x,\dd \alpha) = 
\int_{\IR^d} \big\{w(x)f(x,1)+ 
(1-w(x))f(x,0)\big\}\, \dd x,
$$
this state space is in one-to-one correspondence with the space
${\cal M}_\lambda$ of measures on $\IR^d\times\{0,1\}$ with 
`spatial marginal' Lebesgue measure,
which we endow with the topology of vague convergence. By a slight abuse of notation, 
we also denote the
state space of the process $(w_t)_{t\in\IR_+}$ by 
${\cal M}_\lambda$.

\begin{defn}[SLFV]
\label{slfvdefn}
Let $\mu$ be a finite measure on $(0,\infty)$ and, for each $r\in (0,\infty)$,
let $\nu_r$ be a probability measure on $(0,1]$.
Further, let $\Pi$ be a Poisson point process on 
$\IR^d\times (0,\infty)\times (0,\infty)\times (0,1]$ with intensity measure 
\begin{equation}\label{slfvdrive}
\dd x\otimes \dd t \otimes \mu(\mathrm{d} r)\nu_r(\dd \rho).
\end{equation}
The {\em spatial Lambda-Fleming-Viot process (SLFV)}
driven by \eqref{slfvdrive}, when it exists, is the ${\mathcal M}_\lambda$-valued process 
$(w_t)_{t\in\IR_+}$ with dynamics given as follows.

If $(x,t,r,\rho)\in \Pi$, a reproduction event occurs at time $t$ within the closed 
ball $\cB_r(x)$ of radius $r$, centred on $x$, in which case: 
\begin{enumerate}
\item Choose a parental location $z$ uniformly at random within 
$\cB_r(x)$, 
and a parental type, $\alpha$, according to $w_{t-}(z)$; that is
$\alpha=1$ with probability $w_{t-}(z)$ 
and $\alpha=0$ with probability $1-w_{t-}(z)$.
\item For every $y\in \cB_r(x)$, set 
$w_t(y) = (1-\rho)w_{t-}(y) + \rho\ind_{\{\alpha=1\}}$.
\item For $y\notin \cB_r(x)$, $w_t(y)=w_{t-}(y)$.
\end{enumerate}
We shall refer to $\rho$ as the {\em impact} of the event.
\end{defn}
Before providing conditions under which the process exists, it is convenient to 
introduce the dual process of coalescing lineages that plays the r\^ole for the SLFV played
by the Lambda-coalescents for the (non-spatial) Lambda-Fleming-Viot processes.
The idea is that these lineages trace out the ancestry of a sample from the population.
The dual will also play a crucial r\^ole in establishing the estimates
of Section~\ref{indeplinsection}.

The dynamics of the dual are driven by the same Poisson process of events
$\Pi$ that drives the SLFV. This driving process
is reversible and we shall abuse notation by indexing events by
`backwards time' when discussing our dual. We suppose that
at time $0$, `the present', we sample $k$ individuals from locations
$x_1,\ldots ,x_k$ and we write $\xi_s^1,\ldots ,\xi_s^{N_s}$ for
the locations of the $N_s$ `ancestors' that make up our 
dual at time $s$ before the present. 
\begin{defn}[Dual to the SLFV]
\label{defn of dual}
The coalescing dual process $(\Xi_t)_{t\geq 0}$ is the
$\bigcup_{n\geq 1}(\IR^d)^n$-valued Markov process with dynamics defined as 
follows.  
At each event $(x,t,r,\rho)\in \Pi $:
\begin{enumerate}
\item For each $\xi_{t-}^i\in \cB_r(x)$, independently mark the corresponding 
ancestral lineage with probability $\rho$;
\item if at least one lineage is marked, 
all marked lineages disappear and are replaced by a single
ancestor, whose location is drawn uniformly at random 
from within $\cB_r(x)$.
\end{enumerate}
If no particles are marked, then nothing happens.
\end{defn}
Assuming that the SLFV and its dual exist, the duality is expressed through
the following proposition.
\begin{propn}\label{prop: dual}
The SLFV is
dual to the process $(\Xi_t)_{t\geq 0}$ in the sense that 
for every $k\in \IN$ and $\psi\in C((\IR^d)^k)\cap L^1((\IR^d)^k)$, we have
\begin{align}
\IE_{w_0}\bigg[\int_{(\IR^d)^k} & \psi(x_1,\ldots,x_k)\bigg\{\prod_{j=1}^k w_t(x_j)\bigg\}\, \dd x_1\ldots \dd x_k\bigg] \nonumber\\
& = \int_{(\IR^d)^k} \psi(x_1,\ldots,x_k)\IE_{\{x_1,\ldots,x_k\}}\bigg[\prod_{j=1}^{N_t} w_0\big( \xi_t^j\big)\bigg]\, \dd x_1 \ldots \dd x_k, \label{dual formula}
\end{align}
where the subscripts on the expectations denote the initial values of the 
corresponding processes. In particular, $\IE_{\{x_1,...,x_k\}}$ denotes 
expectation under the distribution of the dual process started from $\Xi_0 = \{x_1,...,x_k\}$.
\end{propn}
In \cite{barton/etheridge/veber:2010}, through a powerful result of \cite{evans:1997}, it is shown that
existence of the SLFV can be deduced from existence of the dual.
In that paper, by assuming that
\begin{equation}
\label{condition for existence 1}
\int_{[0,1]\times (0,\infty )}\rho r^d\nu_r(\dd \rho 
)\mu(\dd r)<\infty,
\end{equation}
one guarantees that started from any finite number of individuals, the
jump rate in the dual is finite, and so Definition~\ref{defn of dual} 
gives rise to a well-defined process. 
Ancestral lineages in the dual process move around according to 
(dependent) compound Poisson processes which can coalesce if they are 
affected by the same event. Although one can write down more 
general conditions under which the SLFV exists, see \cite{etheridge/kurtz:2014},
Condition~\ref{condition for existence 1} is trivially satisfied for the processes
considered below.

\subsection{Main results}
\label{main results}

We are going to extract superBrownian motion from the SLFV through a scaling and a passage
to the limit. There will be two cases, the first leading to the finite variance superprocess
and the second to a superprocess with a stable branching law. 

At the $N$th stage of our scaling, the local population density will
be $K=K(N)$.  We shall denote our scaled SLFV by $w^N$ 
and the population of type $1$ individuals by $X^N=Kw^N$, which is defined 
Lebesgue almost everywhere. We shall think of $X^N$ as a measure-valued 
process and abuse notation by writing, for any Borel measurable $\phi$, 
$$\ldual X_t^N,\phi\rdual =K\int_{\IR^d}\phi(x)w_t^N(x)\dd x
=\int_{\IR^d}\phi(x)
X_t^N(x)\dd x.$$

Before scaling, the SLFV will be driven by a Poisson point process
$\Pi^N$ with intensity $\dd x\otimes \dd t\otimes \mu^N(\dd r)\nu_r(\dd \rho)$. 
Each $(x,t,r,\rho)\in \Pi^N$ signals a reproduction event for 
the scaled process
associated with the quadruple $(\frac{x}{M}, \frac{t}{N}, \frac{r}{M}, \frac{\rho}{J})$,
where $M:= M(N)$, $J:= J(N)$, are some positive 
increasing functions of $N$. In other words, for the scaled process,
time is sped up by a factor $N$, space is
shrunk by $M(N)$, and the impact of each event is reduced by a factor
$J(N)$. Moreover, the local population
density is increased to $K=K(N)$ where $K(N)$ is another increasing
function of $N$.

We shall consider two different scenarios:
\begin{enumerate}
\item{{\bf Fixed radius case:} 
Here we shall take $\mu^N(\dd r)=\delta_r$, independent of $N$, and 
$\nu_r(\dd \rho)=\delta_u$ where
$u\in (0,1]$ is fixed.}
\item{{\bf Variable radius case:} Here we shall take
\begin{align}
\label{variable radius measure}
\mu^N(\dd r) := r^\alpha \ind_{\left\{J^\frac{-1}{\gamma} < r <1 \right\} } \dd r,
\end{align}
where $\alpha$ is a real constant and $\gamma$ is a positive constant. We then take 
$\nu_r :=  \delta_{r^{-\gamma}}$.} 
\end{enumerate}
The lower bound on $r$ in the variable radius case ensures that after scaling
the impact of each
event is at most $1$.

\begin{thm}[Fixed radius case] 
\label{result fixed radius}
In the notation above, in the fixed radius case, suppose that
$X_0^N$ is absolutely continuous with respect to Lebesgue measure, that
the support $\mathrm{supp}(X_0^N)\subseteq D$, where $D$ is a 
compact subset of $\IR^d$ (independent of $N$) and that 
$X_0^N$ converges weakly to $X_0\in {\cal M}_F(\IR^d)$. Moreover suppose that
\begin{enumerate}
\item \label{fixed radius cond1}
$M \to \infty$ as $N\to \infty$,
\item \label{fixed radius cond2}
$\frac{N}{JM^2}\to C_1$ as $N\to\infty$, and
\item 
\label{condition 3 of finite variance case}
$\frac{KN}{J^2 M^d}\to C_2$ as $N\to\infty$,
\end{enumerate}
for some $C_1, C_2 \in (0,\infty)$.
Then if 
\begin{equation}
\label{random walk conditions}
\left\{\begin{array}{rl}
\frac{M}{J} \to 0 &\mbox{if }d=1,\\
\\
\frac{\log M}{J} \to 0 &\mbox{if }d=2,\\
\\
\frac{1}{J} \to 0 &\mbox{if }d\geq3,
\end{array}\right.
\end{equation}
the sequence $\{X^N\}_{N\geq 1}$ converges weakly to 
finite variance superBrownian motion with initial condition $X_0$
and parameters 
$$m=2C_1u r^{d+2}\int_{|x|\leq 1}x^2\dd x, \qquad 
\kappa = \frac{1}{2}C_2u^2|\cB_r|^2.$$
\end{thm}
Conditions~\ref{fixed radius cond1}-\ref{condition 3 of finite variance case}
 guarantee tightness; 
(\ref{random walk conditions}) will ensure that type $1$ is
sufficiently `sparse' that, asymptotically, descendants of different
type 1 `individuals' evolve independently (they don't sense that
total population density is constrained) and we 
recover a branching structure. 
Notice in particular that~(\ref{random walk conditions}), combined 
with Conditions~\ref{fixed radius cond2} 
and~\ref{condition 3 of finite variance case}
implies that $K\to\infty$ as $N\to \infty$.
In Section~\ref{heuristics} we present a heuristic argument which suggests that these
conditions are in some sense optimal.  
We note that the conditions of Theorem~\ref{result fixed radius} are 
analogous to those of \cite{cox/durrett/perkins:2000}, Theorem~1.1. 

It isn't hard to convince oneself that 
if we fix the radius of events, then it is not possible to find a sequence
of impact distributions $\nu^N$ and a scaling under which the limiting 
process is superBrownian motion with a stable branching mechanism of infinite
variance, which is why we turn to $\mu^N(\mathrm{d} r)$.
Theorem~\ref{jonoinfinitelimit} provides conditions 
under which we do then have convergence to superBrownian motion with a 
stable branching mechanism.
Even with our special choice of $\mu^N$, in $d=1$
these are considerably more technical than those in the fixed radius case. 
However, we shall see them emerge in a natural way from our calculations.

\begin{thm}[Variable radius case]\label{jonoinfinitelimit}
In the notation above, suppose that $X_0^N$ is absolutely continuous 
with respect to
Lebesgue measure, that
the support $\mathrm{supp}(X_0^N)\subseteq D$, where $D$ is a 
compact subset of $\IR^d$ (independent of $N$) and that 
$X_0^N$ converges weakly to $X_0\in {\cal M}_F(\IR^d)$. 
We work in the variable radius case with $\mu^N(\mathrm{d} r)$ given 
by~(\ref{variable radius measure})
and $\nu_r=\delta_{r^{-\gamma}}$.
Fix $\beta \in (0,1)$ and take $\alpha$ and $\gamma$ such that
\begin{enumerate}
\item 
$0 < \gamma - d < \frac{1}{1 - \beta}$ and $\gamma >2$ if $d = 1$,
\item
$\alpha + 1 = (\beta +1)(\gamma - d)$.
\end{enumerate}
Suppose that $M\to\infty$ as $N\to\infty$ and that
\begin{enumerate}
\setcounter{enumi}{2}
\item\label{infvarcond1}
$\frac{N}{JM^2} \to C_1$,
\item \label{infvarcond2}
$\frac{N}{J}\left( \frac{K}{JM^d} \right)^{\beta} \to C_2$ and $K\to \infty$,
\item\label{infvarcond3}
$J^{\frac{\gamma - d}{\gamma}} M^{-\frac{2}{\beta}} \to \infty $.
\end{enumerate}
In addition, we require: 
\begin{equation}
\label{infvarrw}
\left\{\begin{array}{rl}
\Big(\frac{1}{M^2}\Big)^{(1-\beta)/\beta}J^{2(1-\beta)(\gamma -1)/\gamma}
\frac{M^2}{J} \to 0 & \mbox{if } d=1;\\
\frac{\log M}{J} \to 0 &\mbox{if } d=2;\\
\frac{1}{J}\to 0 &\mbox{if }d\geq 3.
\end{array}\right.
\end{equation}
Then the sequence $\{X^N\}_{N\geq 1}$ converges to 
superBrownian motion with stable branching law with parameter $\beta$,
initial condition $X_0$, and
$$m=2C_1\int_{|x|\leq 1}x^2 \dd x \int_0^1r^{\alpha+d+2-\gamma}dr,\qquad
\kappa =\frac{C_2}{\gamma -d}\int_0^\infty \big(e^{-\upsilon}+\upsilon
-1\big)\upsilon^{-(\beta+2)}\dd \upsilon.$$
\end{thm}
Once again Conditions~\ref{infvarcond1}--\ref{infvarcond3} guarantee
tightness of the sequence, whereas~(\ref{infvarrw}) ensures `sparsity'.
For $d=1$ the condition in~(\ref{infvarrw}) is not necessary, even our proof shows that it could
be improved a little, but in this form it is easy to check: since $\gamma>2$
and $(1-\beta)(\gamma-1)<1$, we only have to make sure that 
$J\to\infty$ as $N\to\infty$ sufficiently quickly compared to $M$.
\begin{example}
The conditions of Theorem~\ref{jonoinfinitelimit}
are satisfied if 
\begin{enumerate}
\item{$d\geq 2$ and
\begin{enumerate}
\item $\alpha =\beta$, $\gamma =d+1$, 
\item $J=M^\eta$ with $\eta\geq 2\gamma/\beta$, $M^{2+\eta}=N$, 
$K=JM^{d+2/\beta}$.
\end{enumerate}}
\item{$d=1$ and  
\begin{enumerate}
\item{$\beta=3/4$, $\gamma=3$, $\alpha=5/2$,}
\item{$J=M^\eta$ where $\eta>4$, $M^{2+\eta}=N$, $K=JM^{11/3}$.}
\end{enumerate}}
\end{enumerate}
\end{example}

The structure of the proofs of Theorems~\ref{result fixed radius} 
and~\ref{jonoinfinitelimit} will come as no surprise. We establish
tightness of the sequences of rescaled processes and then check that all
limit points satisfy an appropriate martingale problem. The first step
will be to write down the martingale characterisation of the the scaled SLFV
and manipulate it into a form that resembles the desired limit. Tightness
for the fixed radius case is then highly reminiscent of the arguments
in \cite{cox/durrett/perkins:2000}.
To prove tightness in the variable radius case, we 
modify the arguments used in constructing 
superBrownian motion with a stable branching mechanism as the limit of a 
sequence of branching Brownian motions (although the calculations here are
somewhat more involved). In both cases, a key step in identifying the 
limit is to establish
control over the probability that two individuals sampled from the same small
region in the SLFV are close relatives. This is also reminiscent of \cite{cox/durrett/perkins:2000}, being based on estimates for the 
coalescing dual of the SLFV.

\subsection{Heuristics}
\label{heuristics}

Before proceeding to the proofs, let us try to motivate the scalings in 
Theorem~\ref{result fixed radius}
and (at least some of those in) Theorem~\ref{jonoinfinitelimit}.

\subsubsection{Fixed radius}

First consider the fixed radius case. If superBrownian motion really 
is a good approximation
for the type $1$ population, then, in particular, we expect that the 
motion of a single
ancestral lineage in the SLFV should converge to Brownian motion. 
Events that affect regions in which a given lineage lies fall according 
to a Poisson process
with rate proportional to $N$, and, for each such event,
the chance that the lineage is 
affected by it
is $u/J$. Thus the lineage will jump at rate proportional to $N/J$. 
Each jump is 
mean zero, finite variance, and $\cO(1/M)$. In order to obtain a Brownian 
limit, we seek a 
diffusive rescaling; that is $N/(JM^2)$ should converge.

Next recall that in $d\geq 2$ superBrownian motion is a two-dimensional 
object, whereas the support of the SLFV has the same dimension as the 
state space. The scaling of the
impact of each event dictates that an `atom' of 
mass is $\cO(1/J)$, so in order that the number of atoms in a region of
diameter $M$ scale like $KM^2$ (as it would for a two-dimensional
object) we take $KM^2\sim J M^d$, which if $N/(JM^2)$ converges says
that $KN/(J^2M^d)$ should converge.

The large parameter $K$ controls the total population density, but we must
still ensure that the population of rare alleles in our scaled SLFV is
sufficiently `sparse' if we are to recover superBrownian motion.
In the SLFV, the density of the population is strictly regulated, creating a 
strong dependence between the mass born during a reproduction event and that
which dies. In contrast, in superBrownian motion, once born, `individuals' 
reproduce and die independently of one another. In order to ensure that the 
dependence inherent in the SLFV is not apparent to us when we follow just a 
single (rare) type, we should like to know that if we sample individuals from
the same small region they are not likely to be 
close relatives. In this way we 
can guarantee
that individuals are not victims of reproduction events in which their own 
close family reproduces. 

To check whether two individuals sampled from very close to one another
are close relatives, we follow the dual process of ancestral lineages.
We should like them to move apart to a distance of $\cO(1)$ in the scaled
process (without coalescing). 
If both lineages are in the region affected by an event, then the
chance that they are both affected (and therefore coalesce) given that at least one of them
jumps is of order $1/J$. On the other hand, it only takes a finite number of events
in which only one of the lineages jumps before they are sufficiently far apart
that they cannot be affected by the same event and so 
evolve independently. We then 
think of them as making an excursion away
from one another, before they once again come close enough that they are  
susceptible to coalescence. The number of such excursions 
before we see one in which they move apart to a distance of $\cO(1)$ 
(after scaling) 
has mean $\cO(M)$ in $d=1$, $\cO(\log M)$ in $d=2$ and $\cO(1)$ in 
$d\geq 3$. Since at the end of each excursion, the chance that
the lineages will coalesce rather than starting the next excursion is
proportional to $1/J$, we see that our `sparsity' 
conditions~(\ref{random walk conditions}) 
ensure that the probability that they successfully `escape' to a 
distance of $\cO(1)$ from one another tends to one.
This is the intuition underlying the calculations in
Section~\ref{indeplinsection}.

\subsubsection{Variable radii}

Now we turn to the case of variable radii, from which we are trying to 
extract a superBrownian motion with stable branching mechanism.

In the non-spatial setting, \cite{bertoin/legall:2006} recover a stable branching process with parameter $\beta$ from a 
Lambda-Fleming-Viot process in much the same way as we recovered the 
Feller branching process from the Wright-Fisher diffusion in the 
introduction. 
The Lambda-Fleming-Viot process
is driven by a Poisson point process $\widetilde\Pi$ on 
$[0,\infty)\times (0,1]$.
A point $(t,\rho)\in\widetilde\Pi$ signals a 
reproduction event at time $t$ in which a proportion $\rho$ of the 
population is replaced by offspring of a randomly chosen parent.
The intensity of $\widetilde\Pi$ is $\dd t\otimes \Lambda(\dd \rho)/\rho^2$,
where $\Lambda$ is such that the tail
$\Lambda([\varepsilon,1])$ is regularly varying with index $-(1+\beta)$ as
$\varepsilon\to 0$.  
A simple `back of the envelope' calculation illustrates why
Bertoin and Le~Gall's result should hold. To be completely concrete, we
present it in the special case
$\Lambda(\dd \rho)=C(\beta)\rho^{-\beta}(1-\rho)^{\beta}\dd \rho$,
corresponding to the Lambda-Fleming-Viot process which is dual to 
a so-called Beta-coalescent and, as shown in \cite{birkner/blath/capaldo/etal:2005},
is a timechange of the stable branching process.

Once again let $K$ be the total population size and
consider a rare allele that makes up a proportion $w$ of the population. We
are interested in the absolute number, $X=Kw$ of rare alleles. We apply
the infinitesimal generator of $X$ to a test function of the form 
$\exp(-\theta X)$, where $\theta\geq 0$. This yields
\begin{eqnarray*}
\cL^X(e^{-\theta X})&=&
\int_0^1\left\{\frac{X}{K}e^{-\theta(X+K\rho(1-X/K)}+\big(1-\frac{X}{K}\big)
e^{-\theta X(1-\rho)}-e^{-\theta X}\right\}\frac{\Lambda(\dd \rho)}{\rho^2}\\
&=&
C(\beta)\int_0^Ke^{-\theta X}\left\{\frac{X}{K}e^{-\theta v (1-X/K)}
+\big(1-\frac{X}{K}\big)e^{\theta X v/K}-1\right\}
\frac{(1-v/K)^\beta}{(v/K)^{\beta +2}}\frac{1}{K}\dd v\\
&\approx& C(\beta)K^\beta Xe^{-\theta X}\int_0^\infty
\left\{e^{-\theta v}+\theta v -1\right\}
\frac{1}{v^{\beta +2}}\dd v\\
&=&C K^\beta Xe^{-\theta X} \theta^{\beta +1},
\end{eqnarray*}
which we recognise as the infinitesimal generator of the stable branching
process, 
timechanged by a factor proportional to $K^\beta$. 
Recalling the construction of this Lambda-Fleming-Viot process from 
individual based models, for example as in \cite{schweinsberg:2003},
we see that the evolution of a population of size $K$ should be 
compared to the Lambda-Fleming-Viot process on the timescale $1/K^\beta$
(just as we see a factor $1/K$ in the Wright-Fisher diffusion~(\ref{WF spde})).
This precisely cancels the $K^\beta$ we see here.
This calculation confirms that the emergence of the stable branching process 
was dictated by the behaviour of the measure $\Lambda(\dd \rho)$ close to
$\rho=0$.

If we are to extract infinite variance superBrownian motion from 
an SLFV in an analogous way, 
we must have random event radii and, since the spatial motion is
bound up in reproduction events, if in the limit the spatial motion is
to be continuous, the impact 
will depend on the
event size in a nontrivial way. 
In order to make the calculations tractable, 
we fix the impact to be a negative power of the 
radius of events and
we take $\mu^N(\dd r)$ to be a truncated power
law. The purpose of the truncation is two-fold: first, by bounding the 
radii below we force the impact of each event to lie in $(0,1]$;
second, by bounding radii above, we ensure that the jumps of lineages 
in the unscaled SLFV 
have finite moments of all orders.

As for the fixed radius case, we should like the motion of a single 
ancestral lineage to converge to Brownian motion. A lineage will fall in 
the region affected by an event of (scaled) radius $r/M$ at rate
$Nr^d \mu^N(\mathrm{d}r)$ in which case, with probability $\rho/J$ (with
$\rho$ sampled from $\nu_r(\mathrm{d} \rho)$), it
will make a mean zero, finite variance jump of size of order $r/M$. 
Substituting our chosen form of $\mu^N(\mathrm{d} r)\nu_r(\mathrm{d} \rho)$,
we see that in order to obtain a Brownian limit for the motion of 
lineages, we should require convergence of
\begin{align*}
\int_0^\infty \frac{N u(r)r^{d+2}}{JM^2} \; \mu^N(\mathrm{d}r) 
=\int_{J^{-1/\gamma}}^1\frac{N}{JM^2}r^{d+2+\alpha-\gamma}\dd r
.
\end{align*}
Under our conditions on $\alpha$, $\gamma$, this implies that 
$N/(JM^2)$ should converge, just as in the fixed radius case.

We now turn to recovery of the stable branching mechanism. When an event of
radius $r$ falls, the total mass of the offspring in the rescaled 
population process
is $K/(JM^d)$ times
$v(r)=C(d)r^{d-\gamma}$. Such events fall on a given point $x$, and an 
`individual' at that point
is selected as parent of the event, 
at rate $(N/K) r^{\alpha}\dd r \dd x$ (a factor of $C(d)r^d$ in the 
rate at which events of radius $r$ cover the point $x$ has cancelled with
the reciprocal of
the same factor in the probability that $x$ is the point selected uniformly
at random from within the ball to be the location of the parent). 
Observing that $dv(r)=C(d)(d-\gamma)r^{d-\gamma-1}\dd r$ (and thus
substituting for $r^\alpha \dd r$), we see that for a 
given `individual' at $x$, events in which it produces offspring with mass
proportional to $v$ occur at 
rate proportional to
$v^{-(\gamma-d+\alpha+1)/(\gamma-d)}\dd v \dd x$. We need to match 
this with the $v^{-\beta-2}$
of the non-spatial case close to $v=0$, and so we take 
$\alpha+1=(\beta +1)(\gamma-d)$.

A more careful version of the argument in the previous paragraph leads to 
Conditions~\ref{infvarcond2} and~\ref{infvarcond3}.
In fact Condition~\ref{infvarcond2} can be 
understood in terms of the dimension of the limit in much the same way as the
corresponding condition for the fixed radius case.
In the stable case, the Hausdorff dimension of the support of
the limiting superBrownian motion will be $d\wedge (2/\beta)$, so at 
least in high enough dimensions we expect to need $JM^d\sim KM^{2/\beta}$, which
combined with convergence of $N/(JM^2)$ leads to
Condition~\ref{infvarcond2}.
Conditons~(\ref{infvarrw}), which follow from the same
considerations as in the fixed radius case, ensure sufficient `sparsity'.

\section{Martingale characterisation of the process $\{X_t^N\}_{t\geq 0}$}
\label{martingale problem prelimit}

In this section we characterise the distribution of the scaled 
process $\{X_t^N\}_{t\geq 0}$ as a 
solution to a martingale problem. 
It will be convenient to use different formulations of the martingale
problem for
our two scalings. First we need some notation.
Suppose that $\phi\in C_0^3(\IR^d)$
and that $f\in C_0^2(\IR)$.
We use the notation
\begin{align*}
\cL_f^N(\phi)(X_0):= \lim_{t \to 0} \frac{\IE^{X_0}[f(\ldual X^N_t, \phi \rdual) ] 
- f(\ldual X^N_0 , \phi \rdual)}{t}
\end{align*}
for the infinitesimal generator of the measure-valued process $X^N_t$ 
applied to test functions of the form $F(X^N_t)=f(\ldual X^N_t,\phi\rdual)$,
and, with a slight abuse of notation, $\cL^N (\phi)(X^N_0)$ for the 
corresponding quantity when $f(x)\equiv x$.

Notice that we have not assumed that $\phi$ has compact support.
Because $X^N_0$ has 
compact support, the rate at which that support is overlapped by 
a reproduction event is 
bounded, and such an event can only increase the volume of the
support of $X^N$ 
by an amount 
bounded above by the volume of the event (which in turn is uniformly 
bounded). Iterating this argument and comparing to a pure birth process, it
is evident that the support will remain bounded (indeed compact)
up to any finite time, and so the rate of events affecting $X^N$ is bounded. 

Writing down the generator is now standard as our process evolves according to a series of
jumps of finite rate.
Recall that at each point $(x,t,r,\rho)\in\Pi^N$, the scaled process $w^N$ is subject
to the reproduction event associated with the quadruple 
$(\frac{x}{M}, \frac{t}{N}, \frac{r}{M}, \frac{\rho}{J})$. 
At such an event, within the ball $\cB_{r/M}(x)$ we have two possibilities:
$$\left\{
\begin{array}{ll}
w^N_t(y)-w_{t-}^N(y) = -\frac{\rho}{J}w^N_{t-}(y) + \frac{\rho}{J} &\text{ with probability } w_{t-}^N(z), \\ \\
w^N_t(y) -w_{t-}^N(y) = -\frac{\rho}{J}w^N_{t-}(y) &\text{ with probability } 1 - w_{t-}^N(z)
.
\end{array} \right.
$$
We write $\cB_r^M(x)$ for the ball $\cB_{r/M}(x)$ and $|\cB_r^M|=|\cB_r|/M^d$ for its 
volume.
At the time of the first event to affect $X^N$, we have $w^N_{t-}=w^N_0$ and
since, moreover, we only consider $\nu_r$ of the form $\delta_{u(r)}$, 
we find
\begin{multline*}
\cL^N_f(\phi)(X^N_0)
= 
N M^d
\Bigg[
\int_{\IR^d}
\int_0^\infty
\int_{\cB_r^M(x)} 
 \frac{1}{|\cB_r^M|}\\ 
w^N_0(z) f \Bigg( K\frac{u(r)}{J} \int_{\cB_r^M(x)}  \phi(y)  \dd y 
 + K\left(1-\frac{u(r)}{J}\right) \int_{\cB_r^M(x)} \phi(y) w^N_0(y) \dd y  
 + K\int_{\IR^d \backslash \cB_r^M(x)} \phi(y) w^N_0(y) \dd y \Bigg)   \\
+ \big( 1 - w^N_0(z) \big)
f \Bigg( 
K\left(1-\frac{u(r)}{J}\right) \int_{\cB_r^M(x)} \phi(y) w^N_0(y) \dd y  
 + K\int_{\IR^d \backslash {\cB_r^M(x)}} \phi(y) w^N_0(y) \dd y \Bigg) \\
-
f \bigg( K\int_{\IR^d} \phi(y) w^N_0(y) \dd y \bigg) \dd z \; \mu^N(\mathrm{d}r) \dd x \Bigg].
\end{multline*}
Substituting $Kw^N_0=X^N_0$ this becomes
\begin{multline}
\label{infgen general case}
\cL^N_f(\phi)(X^N_0)
= 
N M^d
\Bigg[
\int_{\IR^d}
\int_0^\infty
\int_{\cB_r^M(x)} 
 \frac{1}{|\cB_r^M|}  \\
\frac{X^N_0(z)}{K} f \Bigg( K\frac{u(r)}{J} \int_{\cB_r^M}  \phi(y)  \dd y 
 + \left(1-\frac{u(r)}{J}\right) \int_{\cB_r^M(x)} \phi(y) X^N_0(y) \dd y  
 + \int_{\IR^d \backslash {\cB_r^M(x)}} \phi(y) X^N_0(y) \dd y \Bigg)   \\
+ \big( 1 - \frac{X^N_0(z)}{K} \big)
f \Bigg( 
\left(1-\frac{u(r)}{J}\right) \int_{\cB_r^M(x)} \phi(y) X^N_0(y) \dd y  
 + \int_{\IR^d \backslash {\cB_r^M(x)}} \phi(y) X^N_0(y) \dd y \Bigg) \\
-
f \bigg( \int_{\IR^d} \phi(y) X^N_0(y) \dd y \bigg) \dd z \; \mu^N(\mathrm{d}r) \dd x \Bigg].
\end{multline}
Since $\{X_t^N\}_{t\geq 0}$ is a pure jump Markov process, driven by
a Poisson process of jumps, it follows immediately that for $f$ and $\phi$ as
above, 
$$f(\ldual X_t^N,\phi\rdual)-f(\ldual X_0^N,\phi\rdual)
-\int_0^t\cL_f(\phi)(X_s^N)\dd s$$
defines a mean zero local martingale. In the variable radius case, we shall
exploit this with $f(x)=\exp(-x)$ and non-negative $\phi$. In the 
fixed radius case, the following lemma, which follows immediately
on setting $f(x)=x$ and $f(x) = x^2$, will provide a more
convenient tool.
\begin{lemma}
\label{semimartingale decomp without time}
The quantity $\ldual X_t^N , \phi \rdual $ has the semimartingale decomposition:
\begin{equation}
\ldual X_t^N, \phi \rdual = \ldual X_0^N , \phi \rdual + 
\int_0^t \cL^N (\phi)(X_s^N) \dd s + M^N_t(\phi), 
\end{equation}
where
\begin{multline*}
\cL^N (\phi)(X^N_s) =  \\
\int_{\IR^d} \int_0^\infty \int_{\cB_r^M(x)} \frac{Nu(r) M^{2d}}{J|\cB_r|}\bigg\{ 
 \Big{(} \int_{\cB_r^M(x)}\phi(y) \dd y \Big{)} X^N_s(z) - 
\int_{\cB_r^M(x)} \phi(y) X^N_s(y) \dd y \bigg\} \dd z \; \mu^N(\mathrm{d}r)\dd x.
\end{multline*}
The local martingale $M_t^N(\phi)$ has quadratic variation process
\begin{multline} \label{quadvariationequation}
\ldual M^N(\phi) \rdual_t =  \int_0^t \int _{\IR^d} \int_{\cB_r^M(x)} 
\frac{N M^{2d} u(r)^2}{J^2|\cB_{r}|} \bigg\{ \big(1-\frac{X^N_s(z)}{K}\big) 
\Big{(} \int_{\cB_r^M(x)} \phi(y)X^N_s(y) \dd y \Big{)}^2  \\
 + \frac{X^N_s(z)}{K} \Big{(} \int_{\cB_r^M(x)}K\phi(y) \dd y  - 
\int_{\cB_r^M(x)} \phi(y) X^N_s(y) \dd y \Big{)}^2
\bigg\} \dd z \dd x \dd s
.
\end{multline}
\end{lemma}
It will be convenient to extend the semimartingale decomposition 
to $\phi=\phi(t,x)\in
C_b^{2,3}(\IR\times\IR^d)$. Writing
\begin{multline*}
f(\ldual X_t^N, \phi_t\rdual)- f(\ldual X_0^N, \phi_0\rdual)=
f(\ldual X_t^N, \phi_0\rdual) -f(\ldual X_0^N, \phi_0\rdual)\\
+f'(\ldual X_t^N,\phi_0\rdual)\ldual X_t^N,\phi_t- \phi_0\rdual 
+\cO\big(\ldual X_t^N,\phi_t- \phi_0\rdual^2\|f''\|_\infty\big),
\end{multline*}
we deduce that this extension simply results in 
the additional term
$$f'(\ldual X_0^N,\phi_0\rdual)\ldual X_0^N,\dot\phi_0\rdual$$
in the generator.
In particular, we arrive at the following lemma.
\begin{lemma}
\label{semimartingale decomp with time}
The quantity $\ldual X_t^N , \phi_t \rdual $ has the semimartingale decomposition:
\begin{equation}
\ldual X_t^N, \phi_t \rdual = \ldual X_0^N , \phi_0 \rdual + 
\int_0^t \cL^N (\phi_s)(X_s^N) \dd s + M^N_t(\phi_\cdot), 
\end{equation}
where
\begin{multline*}
\cL^N (\phi_s)(X^N_s) =  \int_{\IR^d} \dot{\phi}_s(y) X^N_s(y) \dd y + \\
\int_{\IR^d} \int_0^\infty \int_{\cB_r^M(x)} \frac{Nu(r) M^{2d}}{J|\cB_r|}\bigg\{ 
 \Big{(} \int_{\cB_r^M(x)}\phi_s(y) dy \Big{)} X^N_s(z) - 
\int_{\cB_r^M(x)} \phi_s(y) X^N_s(y) \dd y \bigg\} \dd z \; \mu^N(\mathrm{d}r)\dd x.
\end{multline*}
\end{lemma}

\setcounter{equation}{0}

\section{Identifying the limit}
\label{limits}

We now turn to identifying the possible limit points of the 
sequence of processes 
$\{X^N\}_{N\geq 1}$ under our chosen scalings, deferring the proof
of tightness of the sequence to Section~\ref{tightness}.

\subsection{Spatial motion}
\label{bounded variation part}

Although our
scaling and limits are quite different, our calculations are
reminiscent of those
of \cite{berestycki/etheridge/veber:2013}.
Our aim is to find an approximate form of the 
martingale problem that is close to that for
superBrownian motion. 

For simplicity we take $\phi$ to be constant in time.
An interchange of integrals
followed by an interchange of the r\^oles of $x$ and $y$ in our notation  
yields
\begin{align*}
\int_{\cB_r^M(x)}\int_{\cB_r^M(x)} \phi(y) & \big( X^N_s(z) - X^N_s(y) \big) \dd y \dd z \\
= & \int_{\cB_r^M(x)}\int_{\cB_r^M(x)} \phi(z) X^N_s(y) \dd y \dd z - 
\int_{\cB_r^M(x)}\int_{\cB_r^M(x)} \phi(y) X^N_s(y) \dd y \dd z \\
= & \int_{\cB_r^M(x)}\int_{\cB_r^M(x)} \big( \phi(z) - \phi(x) + \phi(x) - 
\phi(y) \big) X^N_s(y) \dd y \dd z
.
\end{align*}
We use the Taylor expansion
\begin{align*}
\phi(y) = \phi(x) + \nabla \phi(x) (y-x) +  \frac{1}{2}\sum_{|\alpha|=2} D_\alpha \phi(x)(y-x)^\alpha 
+ \sum_{|\alpha|=3}R_{\alpha}(y)\big((y-x)^{\alpha} \big),
\end{align*}
where $\|R_{\alpha}\|_\infty \leq C\|\phi\|_{C^3}$, with $C$ 
independent of $\phi$, to obtain
\begin{align*}
 & \int_{\cB_r^M(x)}\int_{\cB_r^M(x)} \big( \phi(z)-\phi(x) \big) X_s^M(y) \dd z\; \dd y + 
\int_{\cB_r^M(x)}\int_{\cB_r^M(x)} \big( \phi(x)-\phi(y) \big) X_s^M(y) \dd z\; \dd y \\ 
= &
\int_{\cB_r^M(x)}\int_{\cB_r^M(x)} \bigg( \nabla \phi(x) (z-x) + 
\frac{1}{2}\sum_{|\alpha|=2} D_\alpha \phi(x)(z-x)^\alpha + 
\sum_{|\alpha|=3}R_{\alpha}(z)(z-x)^{\alpha} \bigg) X_s^M(y) \dd z\; \dd y \\
&\phantom{AAAAAAAAAAAAAAAAAAA} 
 + \frac{|\cB_r|}{M^d}\int_{\cB_r^M(x)} \big( \phi(x)-\phi(y) \big) X_s^M(y) \dd y  \\
= & \bigg\{ \frac{C(d) r^{d+2}}{M^{d+2}} \frac{\Delta \phi(x)}{2}
+ {\mathcal O}\big(\frac{\|\phi\|_{C^3}r^{d+3}}{M^{d+3}}\big) \bigg\} 
\int_{\cB_r^M(x)} X^N_s(y) \dd y 
+ \frac{|\cB_r|}{M^d}\int_{\cB_r^M(x)} \big( \phi(x)-\phi(y) \big) X_s^M(y) \dd y,
\end{align*}
where $C(d) := \int_{\cB_1(0)} x^2 \dd x$.
Moreover,
\begin{align}
\nonumber
\int_{\IR^d} \int_{\cB_r^M(x)} \big(\phi(x) - \phi(y) \big) X^N_s(y) \dd y \; \dd x 
= &
\int_{\IR^d} \int_{\IR^d} \ind_{\{|x-y| \leq \frac{r}{M}\}} \big(\phi(x) - 
\phi(y) \big) X^N_s(y)  \dd y \; \dd x \\
\nonumber
= &
\int_{\IR^d } \int_{\IR^d} \ind_{\{|x-y| \leq \frac{r}{M}\}} \big(\phi(x) - 
\phi(y) \big) X^N_s(y)  \dd x \; \dd y \\
= &
\int_{\IR^d} \int_{\cB_r^M(y)} \big(\phi(x) - \phi(y) \big) X^N_s(y) \dd x \; \dd y
.
\label{manipulation}
\end{align}
Again using Taylor's Theorem (and interchanging the role of $x$ and $y$) this
is
\begin{multline*}
\int_{\IR^d} \int_{\cB_r^M(x)} \big(\phi(y) - \phi(x) \big) X^N_s(x) \dd y \; \dd x \\
=
\int_{\IR^d}\int_{\cB_r^M(x)} \bigg( \nabla \phi(x) (y-x) +  
\frac{1}{2}\sum_{|\alpha|=2} D_\alpha \phi(x)(z-x)^\alpha 
 +  \sum_{|\alpha|=3}R_{\alpha}(y)\big((y-x)^{\alpha} \big) \bigg) X^N_s(x) \dd y\; dd x \\
=
\int_{\IR^d} \bigg\{ \frac{C(d) r^{d+2}}{M^{d+2}} \frac{\Delta \phi(x)}{2}
+ \cO\Big(\frac{\|\phi\|_{C^3}r^{d+3}}{M^{d+3}}\Big) \bigg\} X^N_s(x) \dd x
.
\end{multline*}
Combining the above with our expression for $\cL^N(\phi)$ 
from Lemma~\ref{semimartingale decomp without time}, for the fixed radius
case
we obtain
\begin{align*}
\int_0^t \cL^N (\phi)(X^N_s)\dd s
= & 
\frac{NM^{2d}u}{J|\cB_r|} \int_0^t \int_{\IR^d} 
\bigg\{ \frac{C(d) r^{d+2}}{M^{d+2}} \frac{\Delta \phi(x)}{2}+ 
	\cO\Big(\frac{\|\phi\|_{C^3}r^{d+3}}{M^{d+3}}\Big) \bigg\} \\
 & \phantom{AAAAAA}\times 
\bigg( \int_{\cB_r^M(x)} X^N_s(y) \dd y + \frac{|\cB_r|}{M^d}X^N_s(x) \bigg) \dd x \dd s \\
= & \int_0^t \int_{\IR^d} \bigg\{ \frac{C(d)r^{d+2}Nu}{JM^{2}} 
	\frac{\Delta \phi(x)}{2}+ \cO\Big(\frac{Nu\|\phi\|_{C^3}r^{d+3}}{JM^{3}}\Big) \bigg\} \\
&  \phantom{AAAAAA}\times
\bigg( \frac{M^d}{|\cB_r|}\int_{\cB_r^M(x)} X^N_s(y) \dd y + 
X^N_s(x) \bigg) \dd x \dd s
.
\end{align*}
We note that, using the same manipulation as in~(\ref{manipulation}),
$$\int_{\IR^d}\int_{\cB_r^M(x)}\Delta\phi(x)X_s^N(y)\dd y\; \dd x
=\int_{\IR^d}\int_{\cB_r^M(x)}\Delta\phi(y)X_s^N(x) \dd y\; \dd x,$$
and so since $\phi \in C_b^3$, a Taylor expansion yields,
\begin{align*}
\int_0^t \cL^N (\phi)(X^N_s) \dd s
= & 
\int_0^t \int_{\IR^d} \bigg\{ \frac{C(d)r^{d+2}Nu}{JM^{2}} \Delta \phi(x)+ 
	\cO\Big(\frac{Nu\|\phi\|_{C^3}r^{d+3}}{JM^{3}}\Big) \bigg\}
X^N_s(x) \dd x \dd s
\\
= &
\int_0^t \ldual X^N_s, \frac{C(d)r^{d+2}Nu}{JM^{2}} \Delta \phi \rdual
+
	\cO\Big(\frac{Nu\|\phi\|_{C^3}r^{d+3}}{JM^{3}}\Big) \ldual X^N_s, \ind \rdual
\dd s
.
\end{align*}
For the variable radius case, we integrate this expression against
$\nu_r(\mathrm{d} u)\mu^N(\mathrm{d} r)$.
\begin{defn} \label{spatialgenerator}
We denote by $\cA^N$ the operator
\begin{align*}
\cA^N(\phi) := \frac{m(N)}{2}
\Delta \phi
,
\end{align*}
where in the fixed radius case
$$m(N):=\frac{2C(d) Nu r^{d+2}}{JM^{2}},$$ 
and in the variable radius case
\begin{equation}
\nonumber
m(N):=\frac{2C(d)N}{JM^2}\int_0^\infty u(r)r^{d+2}\mu^N(\dd r),
\end{equation}
with $C(d) := \int_{\cB_1(0)} x^2 \dd x$ and $u(r)=r^{-\gamma}$.
\end{defn}
Evidently it is straightforward to extend the calculation above to suitable 
time dependent $\phi_{\cdot}$. We record the result as a lemma.
\begin{lemma} \label{roughgenerator}
For $\phi_s(x):\IR\times \IR^d\to \IR\in C_0^{2,3}$,
\begin{align*}
\int_0^t \cL^N (\phi_s)(X^N_s) & \dd s =  
\int_0^t  \ldual X^N_s, \dot{\phi}_s \rdual + 
	\ldual X^N_s, \cA^N(\phi_s) \rdual \dd s + \zeta_t^N(\phi)
,
\end{align*}
where $|\zeta_t^N(\phi)| \leq
\cO\left(\frac{N\sup_{0\leq s\leq t}\|\phi_s\|_{C^3}}{JM^3}
\int_0^\infty u(r)r^{d+3}\mu^N(\dd r)
\right)
\int_0^t \ldual X^N_s, \ind \rdual \dd s$.
\end{lemma}

\subsection{Fixed radius case}

We now turn to identification of the limit as $N\to\infty$ of the
quadratic variation $\lquad M^N(\phi)\rquad_t$ 
of~(\ref{quadvariationequation}) in the fixed radius case.
Recall that
\begin{multline*}
\lquad M^N(\phi) \rquad_t =
\int_0^t \int _{\IR^d} \int_{\cB_r^M(x)} \frac{N M^{2d} u^2}{J^2|\cB_{r}(x)|}
\bigg\{ \big(1-\frac{X^N_s(z)}{K}\big) \Big{(} \int_{\cB_r^M(x)} 
\phi(y)X^N_s(y) \dd y \Big{)}^2  \\
 + \frac{X^N_s(z)}{K} \Big{(} K\int_{\cB_r^M(x)}\phi(y) \dd y  - 
\int_{\cB_r^M(x)} \phi(y) X^N_s(y) \dd y \Big{)}^2
\bigg\} \dd z \dd x \dd s
.
\end{multline*}
Expanding the brackets we see that
\begin{multline*}
\lquad M^N(\phi) \rquad_t =
\int_0^t \int _{\IR^d} \int_{\cB_r^M(x)} \frac{KN M^{2d} u^2}{J^2 |\cB_r|}
\Bigg[
X^N_s(z) \Big( \int_{\cB_r^M(x)}\phi(y) \dd y\Big)^2
\\ -
2 \frac{X^N_s(z)}{K}
\bigg(
\Big(
\int_{\cB_r^M(x)} \phi(y) \dd y
\Big)
\Big(
\int_{\cB_r^M(x)} \phi(y) X^N_s(y) \dd y
\Big)
\bigg)
\\ +\frac{1}{K}
\Big(
\int_{\cB_r^M(x)} \phi(y) X^N_s(y) \dd y 
\Big)^2
\Bigg]
\dd z \dd x \dd s
,
\end{multline*}
which can be rearranged to yield
\begin{multline}
\label{qv1}
\lquad M^N(\phi) \rquad_t =
\frac{KN |\cB_r|^2 u^2}{J^2 M^d}
\int_0^t \int _{\IR^d} \frac{1}{|\cB_r^M|}
\int_{\cB_r^M(x)} 
\Bigg[
X^N_s(z) \Big(
\frac{1}{|\cB_r^M|}
\int_{\cB_r^M(x)}\phi(y) \dd y\Big)^2
\\ -
2 \frac{X^N_s(z)}{K}
\bigg(
\Big(
\frac{1}{|\cB_r^M|}
\int_{\cB_r^M(x)} \phi(y) \dd y
\Big)
\Big(
\frac{1}{|\cB_r^M|}
\int_{\cB_r^M(x)} \phi(y) X^N_s(y) \dd y
\Big)
\bigg)
\\
+
\frac{1}{K}
\Big(
\frac{1}{|\cB_r^M|}
\int_{\cB_r^M(x)} \phi(y) X^N_s(y) \dd y 
\Big)^2
\Bigg]
\dd z \dd x \dd s
.
\end{multline}
Mimicking the manipulation that gave us~(\ref{manipulation}), and using the
regularity of $\phi$, this can be written
\begin{multline}
\label{quadratic variation}
\lquad M^N(\phi) \rquad_t =
\frac{KN |\cB_r|^2 u^2}{J^2 M^d}
\int_0^t \int _{\IR^d}
\Bigg[
X^N_s(x) \phi^2(x)
-\frac{1}{K}
\Big(
\frac{1}{|\cB_r^M|}
\int_{\cB_r^M(x)} \phi(y) X^N_s(y) \dd y 
\Big)^2
\\
+ \cO\Big(\frac{\|\phi\|_{C^1}^2}{M}\Big) X^N_s(x)
\Bigg]
\dd x \dd s
,
\end{multline}
where we used that if $z\in \cB_r^M(x)$ and $y\in \cB_r^M(x)$, then $|z-y|<2/M$ and 
$X/K\leq 1$ to estimate the error in replacing the second term
in~(\ref{qv1}) by twice the 
third. 
The necessity of 
Conditions~\ref{fixed radius cond1}--\ref{condition 3 of finite variance case} 
of Theorem~\ref{result fixed radius} is already evident 
from~(\ref{quadratic variation}).
This result will be sufficient for the proof of tightness, but more work will
be needed to check that the second term on the right tends to zero as
$N\to \infty$, and thus identify the limiting quadratic variation as 
that corresponding to superBrownian motion. 
The proof of the following lemma, which 
we defer to Section~\ref{indeplinsection}, 
rests on the duality of Proposition~\ref{prop: dual}.
\begin{lemma}
\label{independence requirement fixed radius}
Under the conditions of Theorem~\ref{result fixed radius},
for any $\phi\in C_0^3(\IR^d)$,
$$
\IE\left[\int_0^t \int _{\IR^d}
\frac{1}{K}
\Big(
\frac{1}{|\cB_r^M|}
\int_{\cB_r^M(x)} \phi(y) X^N_s(y) \dd y 
\Big)^2 \dd x \dd s\right] \to 0
$$ 
as $N\to\infty$.
\end{lemma}
If we can prove that our sequence $\{X^N\}_{N\geq 1}$
of processes is tight and that all limit points are martingales, then
granted Lemma~\ref{independence requirement fixed radius}
(and some uniform integrability),
Lemma~\ref{roughgenerator} and~\ref{quadratic variation} allows us to
identify the limit points as solutions to the
martingale problem for finite variance superBrownian motion with $m$ and
$\kappa$ as in the statement of Theorem~\ref{result fixed radius}.

\subsection{Variable radius case}\label{variableradiuslimiting}

We now turn to the variable radius case.
Since, if we are to have convergence to the superBrownian motion
with stable branching,
the quadratic variation of the previous subsection must be unbounded
as $N\to\infty$, 
we instead turn our attention to $\cL^N_f(\phi)$ with
$f(x)=e^{-x}$.
We suppose that $\phi$ is non-negative and, for simplicity, independent of 
time.
Substituting in~(\ref{infgen general case}),
we obtain
\begin{multline*}
\cL_{\exp(- \cdot)}(\phi)(X_0^N)
=
N M^d \exp(-\ldual X^N_0, \phi_0 \rdual)
\int_{\IR^d}
\int_0^\infty
\frac{1}{|\cB_r^M|}
\int_{\cB_r^M(x)} 
\Bigg[
\\
\frac{X^N_0(z)}{K}
\exp \Bigg( -K\frac{u(r)}{J} \int_{\cB_r^M(x)}  \phi(y)  \dd y 
 + \frac{u(r)}{J} \int_{\cB_r^M(x)} \phi(y) X^N_0(y) \dd y  
\Bigg)   \\
+ \big( 1 - \frac{X^N_0(z)}{K} \big)
\exp \Bigg( 
\frac{u(r)}{J} \int_{\cB_r^M(x)} \phi(y) X^N_0(y) \dd y  
\Bigg)
-
1
\Bigg]
\dd z \; \mu^N(\dd r) \dd x
.
\end{multline*}
Taylor expansion of the exponential function yields
\begin{multline*}
\frac{X^N_0(z)}{K}
\exp \Bigg( -K\frac{u(r)}{J} \int_{\cB_r^M(x)}  \phi(y)  \dd y 
 + \frac{u(r)}{J} \int_{\cB_r^M(x)} \phi(y) X^N_0(y) \dd y  
\Bigg)   \\
\phantom{AAAAAAAAAAAAAAAAAA}+ \big( 1 - \frac{X^N_0(z)}{K} \big)
\exp \Bigg( 
\frac{u(r)}{J} \int_{\cB_r^M(x)} \phi(y) X^N_0(y) \dd y  
\Bigg)
- 1
\\
=\frac{X^N_0(z)}{K}
\exp \Big( -K\frac{u(r)}{J} \int_{\cB_r^M(x)}  \phi(y)  \dd y \Big)
\Bigg[1
 + \frac{u(r)}{J} \int_{\cB_r^M(x)} \phi(y) X^N_0(y) \dd y  
 \\+ \cO\Big(\frac{u(r)}{J} \int_{\cB_r^M(x)} \phi(y) X^N_0(y) \dd y\Big)^2  
\Bigg]   \\
+ \big( 1 - \frac{X^N_0(z)}{K} \big)
\Bigg[1+ 
\frac{u(r)}{J} \int_{\cB_r^M(x)} \phi(y) X^N_0(y) \dd y  
+\cO\Big(\frac{u(r)}{J} \int_{\cB_r^M(x)} \phi(y) X^N_0(y) \dd y\Big)^2  
\Bigg]
-1\\
=\frac{X^N_0(z)}{K}
\exp \Big( -K\frac{u(r)}{J} \int_{\cB_r^M(x)}  \phi(y)  \dd y \Big)
-\frac{X^N_0(z)}{K}+\frac{u(r)}{J}\int_{\cB_r^M(x)}\phi(y)X_0^N(y)dy\\
+\frac{X^N_0(z)}{K}\frac{u(r)}{J}\int_{\cB_r^M(x)}\phi(y)X_0^N(y)\dd y
\; \cO\Bigg(\Big(\frac{Ku(r)}{J}\int_{\cB_r^M(x)}\phi(y)dy)\Big)^2\Bigg)\\
+ \cO\Big(\frac{u(r)}{J} \int_{\cB_r^M(x)} \phi(y) X^N_0(y) \dd y\Big)^2  
 \end{multline*}
from which, noting that, under the assumptions of 
Theorem~\ref{jonoinfinitelimit}, $K/(JM^d)\to 0$,
\begin{multline*}
\cL_{\exp(- \cdot)}(\phi)(X_0^N)
=
N M^d \exp(-\ldual X^N_0, \phi_0 \rdual)
\Bigg[
\int_{\IR^d}
\int_0^\infty
\frac{1}{|\cB_r^M|}
\int_{\cB_r^M(x)} \Bigg\{
\\
\frac{X^N_0(z)}{K}
\exp \left( -K\frac{u(r)}{J} \int_{\cB_r^M(x)}  \phi(y)  \dd y 
\right)
+
\frac{u(r)}{J} \int_{\cB_r^M(x)} \phi(y) X^N_0(y) \dd y  - \frac{X^N_0(z)}{K}
\Bigg\}\dd z \; \mu^N(\mathrm{r}) \dd x
\\
+
\int_{\IR^d}
\int_0^\infty
\cO \left(
\frac{u(r)^2}{J^2}
\left(\int_{\cB_r^M(x)}\phi(y) X_0^N(y)\dd y\right)^2
\right)
\mu^N(\mathrm{d}r) \dd x
\Bigg]
.
\end{multline*}
The analogue of Lemma~\ref{independence requirement fixed radius} that
we need in this context is
\begin{lemma}
\label{independence requirement variable radius} 
Under the assumptions of Theorem~\ref{jonoinfinitelimit}, for
$\phi\in C_0^3(\IR^d)$,
$$
\IE\left[\frac{N}{J^2M^d}\int_0^t
\int_{\IR^d}
\int_0^\infty
u(r)^2 r^{2d}
\left(\frac{1}{|\cB_r^M|}\int_{\cB_r^M(x)}\phi(y) X_s^N(y)\dd y\right)^2
\mu^N(\mathrm{d}r) \dd x \dd s\right]
\to 0$$
as $N\to\infty$.
\end{lemma} 
Once again the proof, which relies on duality, is deferred
to Section~\ref{indeplinsection}.

Rearranging the expression for $\cL_{\exp(- \cdot)}(\phi)(X_0^N)$
and reversing the order of integration we find
\begin{multline*}
\cL_{\exp(- \cdot)}(\phi)(X_0^N)
=\exp\left(-\ldual X^N_0, \phi \rdual\right)
\Bigg[
\ldual X^N_0, A^N\phi + B^N\phi \rdual 
\\
+\int_{\IR^d}
\int_0^\infty
\cO \left(
\frac{N u(r)^2 r^{2d}}{J^2M^d}
\right)
\left(\frac{1}{|\cB_r^M|}\int_{\cB_r^M(x)}\phi(y) X_0^N(y)\dd y\right)^2
\mu^N(\mathrm{d}r) \dd x
\Bigg]
,
\end{multline*}
where
\begin{align*}
A^N\phi(x) := &
\frac{NM^d}{K}
\int_0^\infty
\frac{1}{|\cB_r^M|} 
\int_{\cB_r^M(x)} 
\exp
\left(
-K\frac{u(r)}{J} \int_{\cB_r^M(x)}  \phi(y)  \dd y 
\right)
\\
&
\phantom{AAAAAAAAAAAAAAAAAA}-
\exp
\left(
-|\cB_r^M|\frac{Ku(r)}{J} \phi(x)
\right)
\dd z \; \mu^N(\dd r)
,
\\
B^N\phi(x) := &
\frac{NM^d}{K}
\int_0^\infty\Bigg[
\exp
\left(
-|\cB_r^M| \frac{Ku(r)}{J} \phi(x)
\right)
+
|\cB_r^M| \frac{Ku(r)}{J} \phi(x) - 1\Bigg]
\; \mu^N(\dd r)
.
\end{align*}
Consider $A^N\phi(x)$. Again by Taylor expansion, we have
\begin{align}
\nonumber
A^N\phi(x) 
= &
\frac{NM^d}{J} 
\left[\int_0^\infty \frac{1}{|\cB_r^M|} \int_{\cB_r^M(x)} \int_{\cB_r^M(z)} -u(r)\left(\phi(y) - \phi(x) \right) \dd y \dd z \; \mu^N(\mathrm{d}r)
\right]
\\
\nonumber
& +
\int_0^\infty \frac{1}{|\cB_r^M|} \int_{\cB_r^M(x)}
\cO\left(\frac{K M^d N }{J^2}\right)
u(r)^2
\left(
\int_{\cB_r^M(z)} \phi(y) - \phi(x) \dd y
\right)^2
\dd y \dd z \; \mu^N(\mathrm{d}r)
\\
\label{estimate for A}
= &
-\frac{NM^d}{J} 
\left[\int_0^\infty \frac{1}{|\cB_r^M|} \int_{\cB_r^M(x)} \int_{\cB_r^M(z)} u(r)\left(\phi(y) - \phi(x) \right) \dd y \dd z \; \mu^N(\mathrm{d}r)
\right]
\\
\nonumber
& \phantom{AAAAAAAAAAAAAAAAAAAA}+
\int_0^\infty
\cO\left(\frac{K N \|\phi\|_{C^1}^2}{J^2M^{d+2}}\right)
u(r)^2 r^{2d+2}
\; \mu^N(\mathrm{d}r)
,
\end{align}
where the last line follows by observing that 
$\phi(y) - \phi(x) = \cO(r\|\phi\|_{C^1}/M)$
for $y\in \cB_r^M(x)$.

We note that the conditions 
$\gamma - d < \frac{1}{1 - \beta}$ and 
$\alpha + 1 = (\beta +1)(\gamma - d)$ will imply 
$\int_0^\infty u(r)^2 r^{2d + 2} \mu^N(\mathrm{d}r) <\infty$. The calculations
of Section~\ref{bounded variation part}
then allow us to write 
\begin{multline*}
A^N\phi(x) = 
\int_0^\infty\Bigg[
-\frac{C(d)N}{JM^2} u(r) r^{d+2} \Delta\phi
+
\cO\left(\frac{N\|\phi\|_{C^3}}{JM^3} \right) u(r)r^{d+3}
\\
+
\cO \left( \frac{KN\|\phi\|^2_{C^3}}{J^2M^{d+2}} \right) u(r)^2 r^{2d +2}
\Bigg] \mu^N(\mathrm{d}r)
,
\end{multline*}
where $C(d) := \int_{\ball{1}{0}} x^2 \dd x$.

We now specialise to 
$u(r)=r^{-\gamma}$ and 
$\mu^N(\mathrm{d}r) := r^\alpha \ind_{\left\{J^\frac{-1}{\gamma} < r <1 \right\} } \dd r$ as specified in Theorem~\ref{jonoinfinitelimit}.
However, it should be clear that other choices would result in nontrivial 
scaling limits.

Substituting into our expression for $B(x)$ and making the change of 
variable $\upsilon=|\cB_r^M|Ku(r)/J$  
\begin{align}\label{Bintroduction}
B^N\phi(x) := \frac{NM^d}{K(\gamma - d)}
\left( \frac{K|\cB_1|}{JM^d} \right)^{\frac{\alpha + 1}{\gamma - d}}
\int_{\frac{K|\cB_1|}{JM^d}}^{J^{\frac{\gamma - d}{\gamma}}\frac{K|\cB_1|}{JM^d}}
g(\upsilon \phi(x)) \upsilon^{- \left(\frac{\alpha + 1}{\gamma - d} + 1 \right)}
\dd \upsilon
,
\end{align}
where $g(\upsilon) = \exp(-\upsilon) + \upsilon - 1$.
 
Now
$$\int_0^\infty g(\upsilon \phi(x)) \upsilon^{-\beta - 2} \dd \upsilon 
= \kappa_0 \phi(x)^{\beta +1},$$
where $\kappa_0$ is a constant, and so to recover the superBrownian motion
with stable branching law with index $\beta$ in this limit, we choose
$\alpha + 1 = (\beta +1)(\gamma - d)$ (to get the right exponent)
and
$J^{\frac{\gamma - d}{\gamma}} K/(JM^d) \to \infty $, to ensure that
the upper limit of integration tends to infinity. That the lower limit of
integration
$K/(JM^d) \to 0$ was imposed already in order 
for $A^N\phi(x)$ to converge to a nontrivial limit.
To ensure that the term premultiplying the integral in~(\ref{Bintroduction})
is positive and 
finite, we take 
$\gamma-d >0$
and assume that
$\frac{N}{J}\left( \frac{K}{JM^d} \right)^{\beta} \to C_2$.

Granted Lemma~\ref{independence requirement variable radius}, 
we have now shown that under the conditions of 
Theorem~\ref{jonoinfinitelimit},
for any $\phi \in C_0^3(\IR^d)$, as $N\to\infty$,
\begin{align*}
\cL_{\exp(- \cdot)}(\phi)(X^N_t)
\to
\ldual X^N_t, -\frac{m}{2} \Delta \phi + \kappa \phi^{\beta +1} \rdual
\exp\left(-\ldual X^N_t, \phi \rdual \right)
\end{align*}
where $m,\kappa\in (0,\infty)$ are as in the statement
of the Theorem.

\section{Tightness}
\label{tightness}

In this section we turn to the proof of tightness of the sequence
$\{X^N\}_{N\geq 1}$ in the space of 
c\`adl\`ag ${\cal M}_F(\IR^d)$-valued processes.
In the fixed radius case, we shall also check that all limit 
points are actually continuous processes.

To prove tightness, we appeal to a specialised version of Jakubowski's 
general criterion that we have taken from \cite{cox/durrett/perkins:2000}.
For a Borel set $A$, let $X_t^N(A):=\ldual X_t^N,\ind_A\rdual$.
\begin{propn}[\cite{cox/durrett/perkins:2000}, Proposition~3.1] \label{3.1propn}
Let $\Phi \subset C_b(\IR^d)$ be a separating class which is closed 
under addition. A sequence of c\`adl\`ag ${\cal M}_F(\IR^d)$-valued processes 
$\{X^N\}_{N\geq 1}$ 
is tight if and only if the following conditions hold:
\begin{enumerate}
\item
\label{cond 1 propn 3.1}
For each  $T$, $\varepsilon > 0$ there is a compact set 
$K_{T,\varepsilon} \subset \IR^d$ such that
\begin{equation} \label{conpactcontproperty}
\sup_N \IP\left[
\sup_{t\leq T} X^N_t \left(K^c_{T,\varepsilon} \right) > \varepsilon
\right] < \varepsilon.
\end{equation}
\item
\label{cond 2 propn 3.1}
For each $T>0$, $\lim_{H \to \infty} \sup_N \IP [\sup_{t \leq T} \ldual X_t^N, \ind \rdual \geq H] = 0$.
\item
\label{cond 3 propn 3.1}
For each $\phi \in \Phi$, $\{\langle X^N_\cdot, \phi \rangle\}_{N\geq 1}$ 
is tight.
\end{enumerate}
\end{propn}
In fact it is convenient to slightly modify the statements of 
Conditions~\ref{cond 1 propn 3.1} and~\ref{cond 2 propn 3.1}.
\begin{corollary}
\label{limsup corollary}
The conclusion of Proposition~\ref{3.1propn} remains valid if 
Condition~\ref{cond 1 propn 3.1} and~\ref{cond 2 propn 3.1} are replaced by:
\begin{description}
\item[$1'.$]\label{modified cond 1}
For each  $T$, $\varepsilon > 0$ there is a compact set 
$K_{T,\varepsilon} \subset \IR^d$ such that
\begin{equation} 
\limsup_{N\to\infty} 
\IP\left[
\sup_{t\leq T} X^N_t \left(K^c_{T,\varepsilon} \right) > \varepsilon
\right] < \varepsilon,
\end{equation}
\item[$2'$]\label{modified cond 2}
For each $T>0$, 
\begin{equation}
\lim_{H \to \infty} \limsup_{N\to\infty} 
\IP [\sup_{t \leq T} \ldual X_t^N, \ind \rdual \geq H] = 0.
\end{equation}
\end{description}
\end{corollary}
\begin{proof}
We mimic the proof of \cite{ethier/kurtz:1986}, Chapter~3, Corollary~7.4.

Suppose that~(\ref{modified cond 1}) is satisfied, then given 
$\varepsilon>0$, there exists a compact set $K_{T,\varepsilon}^0$ and an
integer $N_0$ such that for all $N> N_0$ 
\begin{equation}
\nonumber
\IP\left[
\sup_{t\leq T} X^N_t \left((K^0_{T,\varepsilon})^c \right) > \varepsilon
\right] < \varepsilon.
\end{equation}
For each $N\leq N_0$ (c.f.~the argument at the beginning of
Section~\ref{martingale problem prelimit}
that $X_t^N$ has compact support for all $t\leq T$), 
there is a set $K_{T,\varepsilon}^N$ such that
\begin{equation}
\nonumber
\IP\left[
\sup_{t\leq T} X^N_t \left((K^N_{T,\varepsilon})^c \right) > \varepsilon
\right] < \varepsilon.
\end{equation}
Set $K_{T,\varepsilon}=\bigcup_{N=0}^{N_0}K_{T,\varepsilon}^N$ and 
Condition~\ref{cond 1 propn 3.1} of Proposition~\ref{3.1propn} is
satisfied.

The proof that $2'$ implies 2 is similar.
\end{proof}

We borrow a convenient formulation of the additional criterion for the
limit points to be continuous from the same paper.
\begin{corollary}[Cox et al.~(2000), Corollary~3.2] 
\label{corollaryforcontinuous}
If a sequence of measure valued processes satisfy the conditions of 
Proposition~\ref{3.1propn} with $\Phi = C_0^\infty (\IR^d)$ and 
for each $\phi \in \Phi$, every limit point of 
$\{\langle X^N_\cdot, \phi \rangle\}_{N\geq 1}$ 
is supported on the space of continuous functions, then 
$\{X^N\}_{N\geq 1}$ is tight and all limit points are continuous.
\end{corollary}

These two results reduce much of the work in proving tightness
to an examination of the 
one-dimensional projections 
$\{\ldual X_\cdot^N,\phi\rdual\}_{N\geq 1}$.
For this we shall make use of the calculations of 
Section~\ref{limits}
which showed, in particular, that $\cL^N(\phi)(X_t^N)$ is close to
$\ldual X_t^N, \cA^N\phi\rdual$ where $\cA^N$, defined in
Definition~\ref{spatialgenerator}, is $m(N)\Delta/2$. 
This will allow us to exploit the following elementary properties of
the heat equation which, for convenience, we record as a lemma.
\begin{lemma} \label{uniform lipschitz}
Suppose that $\psi:\IR^d\to\IR\in C_0^3$ 
and that $v$ is a classical solution to
\begin{equation}
\label{heat equation}
\frac{\partial v}{\partial t} = \frac{m}{2}\Delta v, \qquad
v(0,x) = \psi(x),
\end{equation}
where the diffusion coefficient is a constant $m\in (0,\infty)$.
\begin{enumerate}
\item If $\psi$
is uniformly Lipschitz continuous
with Lipschitz-constant $M$, 
then for all $t$, $v(t,\cdot)$ is Lipschitz continuous with 
Lipschitz constant $M$.
\item For any $k\in\IN$, for each $t$, the $C^k$ norm of $v(t,x)$ (as 
a function of $x$) is
bounded above by that of $\psi$.
\item If $v$ solves the equation with initial data $\psi$
then $\Delta v$ solves the same equation with initial data $\Delta \psi$.
\end{enumerate}
\end{lemma}
\begin{notn}
We denote by $P^{(m)}_t$ the heat semi-group with diffusion coefficient $m$, 
so that the solution to~(\ref{heat equation}) can be written
\begin{align*}
v(t,x)= P^{(m)}_t(\psi)(x) = \IE_x[\psi(B_{mt})],
\end{align*}
where $B_t$ is a standard Brownian Motion.
\end{notn}

\subsection{Verification of Condition~$2'$ of Corollary~\ref{limsup corollary}}

\begin{lemma}
\label{condn 2 of propn}
Under the conditions of either Theorem~\ref{result fixed radius}
or Theorem~\ref{jonoinfinitelimit},
for each $T>0$, 
\begin{enumerate}
\item
\begin{equation}
\label{bound on total mass}
\IE[\ldual X_t^N,\ind\rdual]\leq \ldual X_0^N,\ind\rdual 
\exp\Big(\cO\big(\frac{N}{JM^3}\big)t\Big), \qquad\forall t\leq T;
\end{equation}
\item
$$\lim_{H \to \infty} \sup_N \IP [\sup_{t \leq T} \ldual X_t^N, \ind \rdual 
\geq H] = 0.$$
\end{enumerate}
\end{lemma}
\begin{proof}

\noindent {\bf 1.}
We set $\chi^N(t)=\IE[\ldual X_t^N,\ind\rdual]$. 

Let $\varepsilon>0$ and let $\{h^R\}_{R\geq 1}$ be a 
sequence of smooth, compactly supported
functions on $\IR^d$ such that
$$h^R|_{\ball{R}{0}} = 1,\quad  h^R|_{\ball{2R}{0}^c} = 0, \quad
\Delta h^R \leq \varepsilon,\quad \abs h^R \abs \leq 1,\quad h^R\leq h^{R+1}.$$ 
We can further arrange that $h^R$ are uniformly bounded in $C^3$.

In the notation of Lemma~\ref{roughgenerator},
\begin{align*}
\langle  {X}_t^N, h^R \rangle = 
\langle  {X}_0^N , h^R \rangle + \int_0^{t}  
\langle  {X}_s^N, \cA^N(h^R) \rangle \dd s + 
\zeta_{t}^N (h^R)+   {M}^N_t(h^R).
\end{align*}
Taking expectations,
\begin{align*}
\IE[\langle  {X}_t^N, h^R \rangle] & = 
\langle  {X}_0^N , h^R \rangle + \IE\Big[\int_0^{t}  
\langle  {X}_s^N, \cA^N(h^R) \rangle \dd s\Big] + 
\IE[\zeta_{t}^N(h^R)] \\
& \leq \chi^N(0)+C\|\Delta h^R\|_\infty\int_0^t\chi^N(s) \dd s 
+\cO\Big(\frac{N\|h^R\|_{C^3}}{JM^3}\Big)\int_0^t\chi^N(s)\dd s.
\end{align*}
We have used the fact that, under our assumptions, $m(N)$ 
converges to bound the constant in front of 
the Laplacian, and the bound on the error $\IE[\zeta^N_t(h^R)]$ from
Lemma~\ref{roughgenerator}.
Letting $R\uparrow\infty$ and applying the Monotone Convergence Theorem,
$$\chi^N(t)\leq \chi^N(0)+\Big(C\varepsilon +\cO\big(\frac{N\|h^R\|_{C^3}}{JM^3}\big)\Big)
\int_0^t\chi^N(s)\dd s,$$
and applying Gronwall's inequality and using that $\varepsilon$ was
arbitrary,
$$\chi^N(t)\leq \chi^N(0)\exp\Big(\cO\big(\frac{N}{JM^3}\big)t\Big),$$
as required.

\noindent{\bf 2.}
Define a sequence of stopping times by
$$\tau^N := \inf \{t \geq 0 : \ldual X^N_t, \ind \rdual \geq H \},$$ 
and set $\stopped{X}^N_t := X^N_{t \wedge \tau^N}$. Then
\begin{align*}
\langle  \stopped{X}_t^N, \phi \rangle = 
\langle   \stopped{X}_0^N , \phi \rangle + \int_0^{t \wedge \tau^N}  
\langle   \stopped{X}_s^N, \cA^N(\phi) \rangle \dd s + 
\zeta_{t\wedge \tau^N}^N(\phi) +   \stopped{M}^N_t(\phi),
\end{align*}
where $ \stopped{M}^N_t(\phi)=M^N_{t\wedge \tau^N}$.
Write $\hat{\chi}^N(t) = \IE[\ldual \stopped X^N_t , \ind \rdual]$, and proceed
as above, but with the stopped martingale problem, to obtain
\begin{align*}
\IE[\langle  \stopped{X}_t^N, h^R \rangle ]
= &
\langle   \stopped{X}_0^N , h^R \rangle
+
\IE\left[
\int_0^{t \wedge \tau^N}  \langle   \stopped{X}_s^N, \cA(h^R) \rangle \dd s
\right]
+ \IE[ \zeta_{t\wedge \tau^N}^N(h^R) ]
+ \IE[\stopped{M}^N_t(h^R)]
\\
\leq &
\hat{\chi}^N(0) + C\varepsilon \Big[\int_0^{t\wedge\tau^N}
\ldual\hat{X}_s^N,\ind\rdual\dd s\Big] +
\cO\Big(\frac{NT\|h^R\|_{C^3}}{JM^3}\Big) \IE[\sup_{s \leq t} 
\ldual \stopped{X}^N_s, \ind \rdual]
\\
\leq &
\chi^N(0) + C \varepsilon\int_0^{t}  \chi^N(s) \dd s + 
\cO\left(\frac{HTN\|h^R\|_{C^3}}{JM^3}\right)
.
\end{align*}
The Monotone Convergence Theorem gives
$\IE[\ldual  \stopped{X}_t^N, h^R \rdual ] \to 
\IE[ \ldual \stopped{X}_t^N, \ind \rdual]$
and so Gronwall's inequality implies
\begin{align*}
\IE[ \ldual \stopped{X}_t^N, \ind \rdual]
\leq 
\left(
\IE[ \ldual \stopped{X}_0^N, \ind \rdual] + \cO\left(\frac{HTN\|h^R\|_{C^3}}{JM^3}\right)
\right)
\exp(Ct).
\end{align*}
Markov's inequality now gives
\begin{align}\label{stoppingtimemarkov}
\IP[\ldual \stopped{X}_t^N, \ind \rdual \geq H]
\leq
\frac{\IE[ \ldual \stopped{X}_t^N, \ind \rdual]}{H}
\leq
\left(
\frac{\IE[\ldual X^N_0, \ind \rdual]}{H} 
+
\cO\left( \frac{NT\|h^R\|_{C^3}}{JM^3} \right)
\right)
\exp(Ct)
.
\end{align}
Noting that $\IP[\sup_{t \leq T} \ldual X_t^N, \ind \rdual \geq H]
=
\IP[\ldual \stopped{X}_t^N, \ind \rdual \geq H]$, we can now conclude,
since for any 
$T$, $\varepsilon > 0$ we can choose $H_1$, $N_1$ such that for 
$H \geq H_1$, $N \geq N_1$ and $t\leq T$
the right hand side of~(\ref{stoppingtimemarkov})
is less than $\varepsilon$.
\end{proof}

\subsection{Verification of Condition~$1'$ of Corollary~\ref{limsup corollary}}

\begin{lemma}
\label{condn 1 of propn fixed radius case}
Under the conditions of either Theorem~\ref{result fixed radius}
or Theorem~\ref{jonoinfinitelimit},
for each  $T$, $\varepsilon > 0$ there is a compact set 
$K_{T,\varepsilon} \subset \IR^d$ such that
\begin{equation} \label{compactcont}
\limsup_{N\to\infty} \IP\left[
\sup_{t\leq T} X^N_t \left(K^c_{T,\varepsilon} \right) > \varepsilon
\right] < \varepsilon
\end{equation}
\end{lemma}
\begin{proof}

Let $h \in C^\infty(\IR^d)$ satisfy 
$$h|_{\cB_{1}(0)} = 0,\quad\mbox{ and }\quad  h|_{\cB_{2}(0)^C} = 1.$$ 
Define $h_n(x) := h(x/n)$. It will evidently suffice to show that for
sufficiently large $n$,
\begin{equation}
\label{formulation in h}
\limsup_{N\to\infty} \IP\left[
\sup_{t\leq T} \ldual X^N_t , h_n\rdual > \varepsilon
\right] < \varepsilon .
\end{equation}
We note that the $C^3$ norms of $\Delta h_n$ are uniformly bounded in $n$, 
so that by Lemma~\ref{uniform lipschitz}
the $C^3$ norms of $\Delta P_{t}^{(m)} (h_n)$ are uniformly bounded in 
$n$ and $t$.

With $m(N)$ as in Definition~\ref{spatialgenerator} 
we set $\phi_s := \Delta P_{t-s}^{(m(N))} (h_n)$, so that
\begin{align*}
\dot{\phi_s}(x) + \cA^N(\phi_s)(x)
= 0
.
\end{align*}
By Lemma~\ref{roughgenerator} and Lemma~\ref{uniform lipschitz}
\begin{align}
\label{bound on l applied to heat semigp}
\left| \int_0^t \cL^N (P_{t-s}^{(m(N))} (h_n))(X^N_s) \dd s\right|
=
\cO \left(\frac{N\|h\|_{C^3}}{JM^3}\right) \int_0^t\ldual X^N_s, \ind \rdual 
\dd s
,
\end{align}
and this bound is uniform in $n$.
In particular,
using~(\ref{bound on total mass}),
\begin{align}
\nonumber
\IE[\ldual X_t^N,h_n\rdual]&=
\IE[\ldual X_0^N, P_t^{(m(N))}h_n\rdual]
+\IE\Bigg[
\cO \left(\frac{N\|h\|_{C^3}}{JM^3}\right) \int_0^t\ldual X^N_s, \ind \rdual 
\dd s
\Bigg]
\\
&\leq
\IE[\ldual X_0^N, P_t^{(m(N))}h_n\rdual]
+\cO \left(\frac{N}{JM^3}\right) T
\exp\Big(\cO\big(\frac{N\|h\|_{C^3}}{JM^3}\big)T\Big)
\ldual X^N_0, \ind \rdual
.
\label{bound on expectation}
\end{align}
We recall that
\begin{align}
\label{semimart decomp again}
\ldual X_t^N, h_n \rdual = \ldual X_0^N , h_n \rdual + 
\int_0^t  \ldual X_s^N, \cA^N(h_n) \rdual \dd s + \zeta_t^N(h_n) + M^N_t(h_n)
.
\end{align}
We shall consider each term on the right hand side separately.
The first term will be zero for sufficiently large $n$, since 
$\textrm{supp}(X_0^N)\subseteq D$ for all $N$, where $D\subseteq \IR^d$ is
compact.
The bound~(\ref{bound on expectation}) will help us to control
the integral term. First note that, since $m(N)$ converges,
\begin{align*}
\IE \left[ \sup_{t \leq T} \left| \int_0^t  \ldual X_s^N, \cA^N(h_n) \rdual 
\dd s \right| \right]
\leq &
\IE \bigg[ \int_0^T  \ldual X_s^N, |\cA^N(h_n)| \rdual \dd s  \bigg]
\\
\leq &
C \IE \bigg[ \int_0^T  \ldual X_s^N, \ind_{\cB_{n-1}(0)^C} \rdual \dd s 
 \bigg]
,
\end{align*}
where $C$ is independent of $n$ and $N$. 
Now set $m:= \lfloor \frac{n-1}{2} \rfloor$.
\begin{align}
\nonumber
\IE \bigg[ \int_0^T  \ldual X_s^N, \ind_{\cB_{n-1}(0)^C} \rdual \dd s
 \bigg] 
\leq 
\IE \bigg[ \int_0^T  \ldual  X_s^N, h_m \rdual \dd s \bigg] &
\leq 
\int_0^T  \IE[\ldual X_s^N, h_m \rdual ]
\dd s
\\
\label{use of bound on l}
\leq 
\int_0^T  \ldual X^N_0, P_{t}^{(m(N))} (h_m)\rdual  \dd s
+\cO &\left(\frac{N\|h_m\|_{C^3}}{JM^3}\right) T^2
\exp\Big(\cO\big(\frac{N}{JM^3}\big)T\Big)\ldual X_0^N,\ind\rdual,
\end{align}
where we used~(\ref{bound on expectation}) to 
obtain~(\ref{use of bound on l}). Under the assumptions of 
Theorem~\ref{result fixed radius} or Theorem~\ref{jonoinfinitelimit}, 
$N/(JM^2)\to C_1$ and $M\to\infty$ as $N\to\infty$ and so
we can take $N$ sufficiently large that the error term is less 
than $\varepsilon^2/32$.

Denoting standard Brownian motion by $\{B_t\}_{t\geq 0}$, consider now
\begin{align}
\nonumber
\int_0^T  \ldual X_0^N, P_{s}^{(m(N))} (h_m) \rdual \dd s
= & \int_0^T  \ldual X^N_0, \IE^x[ h_m(B_{m(N)s})] \rdual \dd s \\
\nonumber
\leq &
\int_0^T  \ldual X^N_0, \IE^x[ \ind_{\compliment{\cB_m(0)}}(B_{m(N)s})]\rdual
\dd s \\
\nonumber
\leq &
T \Big[ X_0^N( \compliment{\cB_{m/2}(0)} ) + X_0^N(\cB_{m/2}(0))\sup_{s \leq T}
\IP^0[|B_{m(N)s}| \geq m/2] \Big] \\
\leq &
T \Big[ X_0^N( \compliment{\cB_{m/2}(0)} ) + 
X_0^N(\IR^d)\sup_{s \leq T}\IP^0[|B_{m(N)s}| \geq m/2] \Big],
\label{doubling trick}
\end{align}
which (by our assumptions that $\mathrm{supp} (X_0^N)\subseteq D$ for all $N$
and that $m(N)$ converges)  
tends to $0$ as $m \to \infty$ uniformly in $N$, so in particular is
$<\varepsilon^2/32$ for sufficiently large $m$. 

Combining the estimates above, we have shown that
given $\varepsilon >0$, there exist $N_0$, $n_0$ such that for $N\geq N_0$ and
$n\geq n_0$,
$$\IE \left[ \sup_{t \leq T} \left| \int_0^t  \ldual X_s^N, \cA^N(h_n) 
\rdual \dd s \right|  \right] <\varepsilon^2/16.$$

From Lemma~\ref{roughgenerator}, 
$$\sup_{t\leq T}|\zeta_t^N(h_n)|\leq \cO\Big(\frac{N\|h_n\|_{C^3}}{JM^3}\Big)\int_0^T
\ldual X_s^N,{\mathbf 1}
\rdual \dd s$$
and we can estimate the right hand side 
through~(\ref{bound on total mass}) from 
Lemma~\ref{condn 2 of propn}
and see that for sufficiently large $N$ it is bounded above 
by $\varepsilon^2/16$.

In order to control the martingale term we shall control 
$\IE[|M_T^N(h_n)|]$ and then apply Doob's inequality. 

For sufficiently large $N$, 
the argument that gave us~(\ref{doubling trick}) allows us to 
bound~(\ref{bound on expectation})
with $t=T$ by $\varepsilon^2/8$.
Rearranging~(\ref{semimart decomp again})
and using the triangle inequality, given $\varepsilon >0$, 
there exist $N_0$, $n_0$, such that uniformly in $n\geq n_0$, for 
$N\geq N_0$,
$\IE[|M^N_T(h_n)|] \leq \varepsilon^2/4$.
Therefore by Doob's maximal inequality, for any such $n$ and $N$, we have that
\begin{align*}
\IP[ \sup_{t \leq T} |M^N_t(h_n)| > \frac{\varepsilon}{2} ] 
\leq \frac{2\IE[|M^N_T(h_n)|]}{\varepsilon}<\frac{\varepsilon}{2}.
\end{align*}
Combining the estimates above with another application of Markov's 
inequality to the sum of the remaining terms on the 
right hand side of~(\ref{semimart decomp again})
we have verified~(\ref{formulation in h})
and the proof is complete.
\end{proof}

\subsubsection{Tightness of projections: fixed radius case}

We shall verify Condition~\ref{cond 3 propn 3.1}
of Proposition~\ref{3.1propn} separately for our two cases. In the
fixed radius case it is relatively straightforward as the quadratic
variation of the martingale part of our semimartingale 
decomposition~(\ref{semimart decomp again}) will remain bounded as
$N\to\infty$. In this section we shall deal with that case and 
moreover check the conditions of 
Corollary~\ref{corollaryforcontinuous} so that we can deduce that 
all the limit points of $\{X^N\}_{N\geq 1}$ are in fact continuous 
processes. We shall exploit some standard results that
we record here for convenience.
\begin{propn}[\cite{jacodshiryaev2003limittheorems}, Chapter VI, Part of
Proposition 3.26] 
\label{tightnessofprojections}
A sequence $\{Y^N\}_{N \geq 1}$ of c\`adl\`ag $\IR^d$-valued processes
is tight and all limit points of the sequence of laws of $Y^N$ are 
laws of continuous processes if and only if the following two conditions 
are satisfied:
\begin{enumerate}
\item
$\forall T \in \IN$, $\varepsilon > 0$, $\exists N_0 \in \IN$, 
$\Gamma \in \IR$ such that
\begin{align*}
N \geq N_0 \Rightarrow \IP \bigg[ \sup_{t \leq T} | Y^N_t | > \Gamma \bigg] \leq \varepsilon.
\end{align*}
\item
$\forall T \in \IN$, $\varepsilon >0$, $\eta > 0$, $\exists N_0 \in \IN$, $\delta > 0$ such that
\begin{align*}
N \geq N_0 \Rightarrow \IP \bigg[ \sup_{0 \leq t \leq T- \delta} \Big( \sup_{s_1,s_2 \in [t, t + \delta]}
|Y^N_{s_1} - Y^N_{s_2}| \Big) > \eta \bigg] \leq \varepsilon.
\end{align*}
\end{enumerate}
\end{propn}
\begin{thm}[\cite{jacodshiryaev2003limittheorems}, Chapter VI, 
Theorem~4.13] \label{tightqv}
If for each $N$, $M^N$ is a locally square integrable local martingale, then 
sufficient conditions for the sequence $\{M^N\}_{N\geq 1}$
to be tight are:
\begin{enumerate}
\item
The sequence $\{M^N_0\}_{N\geq 1}$ is tight.
\item
The sequence $\{\lquad M^N \rquad\}_{N\geq 1}$ is tight with all 
limit points being continuous.
\end{enumerate}
\end{thm} 
\begin{propn}[\cite{jacodshiryaev2003limittheorems}, Chapter VI, Part of
Proposition 3.26] 
\label{JS3.26}
A sequence $\{Y^N\}_{N \geq 1}$ is tight and all limit points of the 
sequence of laws of $Y^N$ are laws of continuous processes if and only 
if $\{Y^N\}_{N \geq 1}$ is tight and for all $T > 0$, $\varepsilon > 0$,
\begin{align*}
\lim_{N \to \infty} \IP[\sup_{t \leq T} |Y^N_t - Y^N_{t-}| > \varepsilon ] = 0
.
\end{align*}
\end{propn}
\begin{lemma}
\label{lemma for projections}
Under the assumptions of Theorem~\ref{result fixed radius}, for each
$\phi\in C_0^3(\IR^d)$, $\{\ldual X^N_\cdot,\phi\rdual\}_{N\geq 1}$ is
tight and all the limit points are continuous.
\end{lemma}
\begin{proof}
Consider $Y_t^N := \ldual X^N_t,\phi \rdual$ with $\phi \in C_0^3(\IR^d)$.
Since $|Y_t^N| \leq \|\phi\|_{\infty} \ldual X^N_t, \ind \rdual$, the first
condition of Proposition~\ref{tightnessofprojections} 
follows immediately from our verification of Condition~$2'$
of Corollary~\ref{limsup corollary}.
To check the second condition of Proposition~\ref{tightnessofprojections},
we once again use the semimartingale decomposition
\begin{align*}
\ldual X^N_t, \phi \rdual
=
\ldual X^N_0, \phi \rdual
+
\int_0^t
\ldual X^N_s, \cA^N(\phi) \rdual
\dd s
+
\zeta^N_t (\phi)
+
M^N_t(\phi)
.
\end{align*}
Evidently
\begin{align*}
\left|
\int_0^{s_1}
\ldual X^N_s, \cA^N(\phi) \rdual
\dd s
-
\int_0^{s_2}
\ldual X^N_s, \cA(\phi) \rdual
\dd s
\right|
\leq &
C(\phi) |s_1 - s_2| \sup_{t \leq T}\ldual X^N_t, \ind \rdual
\end{align*}
and combining with
$$\abs \zeta^N_t(\phi) \abs \leq 
\cO\left(\frac{Nur^{d+3}\|\phi\|_{C^3}}{JM^3}\right) t \sup_{s\leq t} 
\ldual X^N_s, \ind \rdual$$ 
and~\eqref{stoppingtimemarkov} we see that
under the conditions of Theorem~\ref{result fixed radius} both these
terms satisfy Condition~2 of Proposition~\ref{tightnessofprojections}
and our problem is reduced to 
showing tightness with continuous limits of the martingale part.
Theorem~\ref{tightqv} tells us that it suffices to consider the
quadratic variation.

Recall from~(\ref{quadratic variation}) that
\begin{multline}\label{quadraticvariationfixedscaled}
\lquad M^N(\phi) \rquad_t
=
\int_0^t
 \int _{\IR^d}  \frac{K N |\cB_{r}| u^2}{J^2 M^{d}}
\bigg[ \bigg{(} \phi(x)^2 X^N_s(x) - \frac{1}{K}
\Big( \frac{M^d}{|\cB_{r}|} 
\int_{\cB_r^M(x)}\phi(y) X^N_s(y) \dd y \Big{)} ^2 \bigg)
\\
+ \cO\big( \frac{\|\phi\|_{C^1}^2}{M} \big) X^N_s(x)
\bigg]  \dd x \; \dd s 
\\
\leq 
\int_0^t
 \int _{\IR^d}  \frac{K N |\cB_{r}| u^2}{J^2 M^{d}}
\bigg[\ldual X_s^N,\phi^2\rdual 
+ \cO\big( \frac{\|\phi\|_{C^1}^2}{M} \big) \ldual X^N_s,1\rdual\bigg] ds
.
\end{multline}
Using Proposition~\ref{tightnessofprojections} and our
bounds on $\ldual X_s^N,\ind\rdual$ from 
Lemma~\ref{condn 2 of propn},
tightness of $\{\lquad M^N(\phi)\rquad\}_{N\geq 1}$ is immediate.

We now conclude through an application of Proposition~\ref{JS3.26}.
The probability that the support of $X^N$, which is compact, is 
simultaneously overlapped by two events of the SLFV is zero, and $\phi$
is bounded, so
$$\sup_{t \leq T} |Y^N_t - Y^N_{t-}| \leq \frac{C(d)K}{J M^d}, \quad
\IP-a.s.$$
\end{proof}
\begin{propn}
Under the conditions
of Theorem~\ref{result fixed radius},
$\{X^N\}_{N\geq 1}$ is tight and all limit points are continuous.
\end{propn}
\begin{proof}
This is now immediate from Corollary~\ref{corollaryforcontinuous}.
\end{proof}

\subsubsection{Tightness of projections: variable radius case}

Our proof in the fixed radius case breaks down in the variable 
radius setting since the quadratic 
variation of the martingale part of $\ldual X_\cdot^N,\phi\rdual$ will
grow without bound as $N\to\infty$. Instead we exploit the fact that
the $(1+\theta)$th moments of $\sup_{0\leq t\leq T}\ldual X_t^N,\phi\rdual$
will remain bounded for any $\theta <\beta$. 
We follow \cite{dawson:1993} who used essentially the same argument  
to show that branching Brownian motion with stable 
branching law, suitably rescaled, converges to superBrownian motion
with stable branching, although
there are some extra layers of estimation in our setting.

We exploit the following elementary lemma, which we learned from a preliminary
version of \cite{dawson:1993}, but we record here since it does not appear in the 
published version.

\begin{lemma} \label{lemmaDawsonabcde}
There exist $c_1, c_2, c_3\in (0,\infty)$ such that
\begin{enumerate}
\item \label{lemmaDawsonconda}
$1 - x + \frac{x^2}{2} - \exp(-x) \geq 0$, for $x \geq 0$.
\item \label{lemmaDawsoncondb}
$\frac{1}{x} \left(1 - x + \frac{x^2}{2} - \exp(-x)\right) \geq c_1 > 0$ for $x \geq 2$.
\item \label{lemmaDawsoncondc}
$0 \leq \exp(-x) +x - 1 \leq c_2 x^{1 + \beta}$, for $x \geq 0$ where $0 < \beta \leq 1$.
\item \label{lemmaDawsoncondd}
For a non-negative random variable $Y$ and 
$\theta \neq -1$, 
$$\IE[ Y^{1 + \theta}] \leq 2 + 
(1 + \theta) \int_1^\infty y^\theta \IP[Y \geq y] \dd y.$$
\item \label{lemmaDawsonconde}
For a non-negative random variable $Y$ and any $y \geq 1$, 
$$\IP[ Y \geq y] \leq c_3 y\int_0^{\frac{2}{y}} 
\IE[ \exp(-\lambda Y) - 1 + \lambda Y ] \dd \lambda .$$
\end{enumerate}
\end{lemma}
\begin{proof}(sketch)
For \ref{lemmaDawsoncondd}, note that
\begin{align*}
x^{1+\theta} = 1 + (1 + \theta) \int_1^x y^\theta \dd y
,
\end{align*}
and so
\begin{align*}
\IE[ Y^{1 + \theta} ] 
&
\leq 1+ \int_1^\infty \left(  1 + (1 + \theta) \int_1^x y^\theta \dd y \right) \mathrm{d}\IP[Y \leq x]
\\
&
\leq 2 + \int_1^\infty  (1 +\theta) \int_1^x y^\theta \dd y \; \mathrm{d}\IP[Y \leq x]
\\
&
= 2 + (1 +\theta)\int_1^\infty y^\theta \IP[Y \geq y] \dd y
,
\end{align*}
where the last line follows from reversing the order of integration.

For \ref{lemmaDawsonconde}, observe that
\begin{align*}
y\int_0^{\frac{2}{y}} \IE[ \exp(-\lambda Y) - 1 + \lambda Y ] \dd \lambda
& =
\int_0^\infty y \int_0^{\frac{2}{y}} \Big(   \exp(-\lambda x) - 1 + \lambda x\Big) 
\dd \lambda \dd \IP[Y \leq x] 
\\
& =
\int_0^\infty \frac{y}{x} \left( 1 - \frac{2x}{y} + \frac{1}{2} \left(\frac{2x}{y}\right)^2 - \exp\left(- \frac{2x}{y}\right)
\right)  \mathrm{d} \IP[Y \leq x]
\\
& \geq 
\int_y^\infty \frac{y}{x} \left( 1 - \frac{2x}{y} + \frac{1}{2} \left(\frac{2x}{y}\right)^2 - \exp\left(- \frac{2x}{y}\right)
\right)  \mathrm{d} \IP[Y \leq x]
\\
& \geq 2 c_1 \IP[Y \geq y]
,
\end{align*}
where we have used \ref{lemmaDawsonconda} and \ref{lemmaDawsoncondb}.
\end{proof}

In the classical setting of Dawson, at this stage one exploits the fact 
that the approximating processes are branching processes, in exact duality with
the solution to a deterministic evolution equation. This is no longer true
in our setting, so instead our first task is to write down an {\em approximate}
evolution equation, whose solution we denote by $v_\phi^N$, with the property that 
\begin{align*}
\IE[\exp( -\ldual X^N_t, \phi \rdual ] = 
\exp(- \ldual X^N_0, v_\phi^N(t,\cdot) \rdual)
+ \cO(\varepsilon^N) ||\phi||_\infty^2\ldual X_0^N,\ind\rdual
,
\end{align*}
where $\varepsilon^N \to 0$ as $N \to \infty$.

From our calculations in Section~\ref{variableradiuslimiting} (in particular
equations~(\ref{estimate for A} and~(\ref{Bintroduction})), and assuming
the result of Lemma~\ref{independence requirement variable radius}, we can
write
\begin{equation}
\label{approximate laplace transform}
\cL_{\exp(- \cdot)}(\phi)(X_0^N)
=\exp\left(-\ldual X^N_0, \phi \rdual\right)
\Bigg[
\ldual X^N_0, -L^N\phi + B^N\phi \rdual + \cO(\widetilde{\varepsilon}^N) ||\phi||_\infty^2\ldual X_0^N,\ind\rdual\Bigg]
\end{equation}
with $\widetilde{\varepsilon}^N\to 0$, where
$$L^N\phi(x)=
\frac{NM^d}{J} 
\left[\int_0^\infty \frac{1}{|\cB_r^M|} \int_{\cB_r^M(x)} \int_{\cB_r^M(z)} u(r)\left(\phi(y) - \phi(x) \right) \dd y \dd z \; \mu^N(\mathrm{d}r)
\right]
$$
and 
\begin{align*}
B^N\phi(x) := \frac{NM^d}{K(\gamma - d)}
\left( \frac{K|\cB_1|}{JM^d} \right)^{\frac{\alpha + 1}{\gamma - d}}
\int_{\frac{K|\cB_1|}{JM^d}}^{J^{\frac{\gamma - d}{\gamma}}\frac{K|\cB_1|}{JM^d}}
g(\upsilon \phi(x)) \upsilon^{- \left(\frac{\alpha + 1}{\gamma - d} + 1 \right)}
\dd \upsilon
,
\end{align*}
Now let us define $v^N_\phi(t,x)$ by
\begin{align}
\label{approximate evolution}
\begin{cases}
\frac{\partial v^N_\phi(t,x)}{\partial t} & = L^Nv^N_\phi(t,\cdot)(x) - 
B^Nv^N_\phi(t,\cdot)(x)
,
\\
v^N_\phi(0,x) & = \phi(x).
\end{cases}
\end{align}
This is the evolution equation corresponding to a superprocess in which the 
spatial motion is the compound Poisson process with generator $L^N$
and the branching mechanism is determined by $B^N$. The existence of such a 
process (which then, by duality, guarantees the uniqueness of the solution
of the evolution equation) follows from, for example, \cite{el-karoui/roelly:1991}.

Let us write $\{S_t^N\}_{t\geq 0}$ for the semigroup generated by $L^N$. 
Then the solution to~(\ref{approximate evolution}) can be written
\begin{align}\label{vcharacterisationintermsofs}
v^N_\phi(t,x) := S^N_t(\phi)(x) - \int_0^t S^N_{t-s} 
\left( B^Nv^N_\phi(s,\cdot)\right)(x) \dd s
.
\end{align}
An immediate consequence of~(\ref{approximate laplace transform}) is 
that we have an approximate duality between 
$X^N$ and $v_{\phi}^N$ since
\begin{align*}
\left|\cL_{\exp(- \cdot)}(v_\phi^N(t-s,\cdot)(X_s^N)\right| &\leq
\cO(\varepsilon^N) ||v_\phi^N(t-s,\cdot)||_\infty^2\ldual X_s^N,\ind\rdual
\exp\left(-\ldual X^N_s, v_\phi^N(t-s,\cdot) \rdual\right)\\
& \leq
\cO(\varepsilon^N) ||v_\phi^N(t-s,\cdot)||_\infty^2\ldual X_s^N,\ind\rdual,
\end{align*}
and so integrating over $[0,t]$,
\begin{eqnarray}
\nonumber
\IE\big[
\exp\left(-\ldual X^N_t, \phi \rdual\right)\big]
&=&
\exp\left(-\ldual X^N_0, v_\phi^N(t,\cdot) \rdual\right) +
\cO(\varepsilon^N) \int_0^t\|v^N_\phi(t-s,\cdot)\|_\infty^2\dd s
\IE[\ldual X_0^N,\ind\rdual]
\\
\label{approximate duality relation}
&=&
\exp\left(-\ldual X^N_0, v_\phi^N(t,\cdot) \rdual\right) +
\cO(\varepsilon^N) ||\phi||_\infty^2\ldual X_0^N,\ind\rdual,
\end{eqnarray}
where we used Lemma~\ref{condn 2 of propn}
to replace $\IE [ \int_0^t\ldual X_s^N,\ind\rdual\dd s ]$ by 
$\ldual X_0^N,\ind\rdual$
at the expense of replacing $\widetilde{\varepsilon}^N$ by
$\varepsilon^N$, but still with $\varepsilon^N\to 0$ as $N\to\infty$, and 
we used~(\ref{vcharacterisationintermsofs}) 
to see that $\|v_\phi^N(t-s)\|_\infty \leq\|\phi\|_\infty$.
In particular, applying this with $\lambda\phi$ in place of $\phi$ 
and differentiating with
respect to $\lambda$ at $\lambda=0$ gives
\begin{align*}
\IE[ \ldual X^N_t, \phi\rdual] = \ldual X^N_0, S^N_t(\phi) \rdual
.
\end{align*}

The key step in proving tightness is the following:
\begin{lemma} \label{supboundthetamomentlemma}
For $\{X^N\}_{N\geq 1}$ satisfying the conditions of Theorem~\ref{jonoinfinitelimit} and 
$\phi \in C^\infty_0$ we will have that, 
for any fixed $0 < \theta < \beta$,
\begin{align*}
\IE[ \sup_{s \leq T} \ldual X^N_s, \phi \rdual^{1 + \theta} ] \leq H
,
\end{align*}
where $H$ is independent of $N$.
\end{lemma}
\begin{proof}
In the proof of this lemma, $C$ is a constant which is independent of $N$,
but may change from line to line.
First observe that, using Lemma~\ref{lemmaDawsonabcde} 
parts~\ref{lemmaDawsoncondd} and~\ref{lemmaDawsonconde},
\begin{multline}
\IE[ \ldual X^N_t, \phi \rdual^{1 + \theta} ]
\leq
2 + C \int_1^\infty y^{1 + \theta} \int_0^{\frac{2}{y}} 
\IE[ \exp(- \ldual X^N_t, \lambda \phi \rdual) - 1 + \ldual X^N_t, 
\lambda \phi \rdual ] \dd \lambda \dd y
\\
=
2 + C \int_1^\infty y^{1 + \theta} \int_0^{\frac{2}{y}}\Big[ 
\exp(- \ldual X^N_0, v^N_{\lambda \phi}(t,\cdot) \rdual) - 1 +  
\ldual X^N_0, S^N_t(\lambda \phi) \rdual 
\\
 + 
\cO(\varepsilon^N)\|\lambda \phi\|_\infty^2t\ldual X_0^N,\ind\rdual
\Big] 
\dd \lambda \dd y
,
\label{theta moment bound}
\end{multline}
where we have used~(\ref{approximate duality relation} 
and~(\ref{bound on total mass}) and
the constant $C$ is determined by those in Lemma~\ref{lemmaDawsonabcde}.
Using Part~\ref{lemmaDawsoncondc} of Lemma~\ref{lemmaDawsonabcde},
\begin{multline*}
\exp(- \ldual X^N_0, v^N_{\lambda \phi}(t,\cdot) \rdual) - 1 +  
\ldual X^N_0, S^N_t(\lambda \phi) \rdual
\\
\leq
c_2 \ldual X^N_0, \lambda S^N_t(\phi) \rdual^{1 +\beta} 
+ \exp( - \ldual X^N_0, v^N_{\lambda \phi}(t,\cdot) \rdual ) - 
\exp(- \ldual X^N_0, S^N_t(\lambda \phi) \rdual
.
\end{multline*}
The first term is at most 
$c_2\|\lambda\phi\|_\infty^{1+\beta}\ldual X_0^N,\ind\rdual^{1+\beta}$.
To control the remaining terms, note that
\begin{align*}
\exp( - \ldual X^N_0, v^N_{\lambda \phi}(t,\cdot) \rdual ) - 
\exp(- \ldual X^N_0, S^N_t(\lambda \phi)(\cdot) & \rdual
\leq \ldual X^N_0, |v^N_{\lambda \phi}(t,\cdot ) - S^N_t(\lambda \phi)(\cdot) | \rdual
\\
& =
\ldual X^N_0, \Big| \int_0^t S^N_{t-s} 
\left( B^Nv^N_{\lambda \phi}(s,\cdot) \right) \dd s \Big| \rdual
,
\end{align*}
(where we have used \eqref{vcharacterisationintermsofs}).
Now for non-negative $\phi$, under the assumptions of 
Theorem~\ref{jonoinfinitelimit}, using~(\ref{Bintroduction}),
\begin{align*}
B^N\phi(x)
\leq & 
\frac{NM^d}{K(\gamma - d)} \left(\frac{K \phi(x)}{JM^d}\right)^{\frac{\alpha +1}{\gamma - d}} \int_0^\infty
\left(\exp(-v) + v - 1 \right) v^{-\frac{\alpha + 1}{\gamma - d} - 1} \dd v
\\
\leq &
C \phi(x)^{1 +\beta}
,
\end{align*}
so that using~(\ref{vcharacterisationintermsofs}) again,
\begin{align*}
\ldual X_0^N, \Big|\int_0^t S^N_{t-s} 
\left( B^Nv^N_{\lambda \phi}(s,\cdot) \right) \dd s \Big|\rdual
& \leq 
C \ldual X_0^N,\int_0^t S^N_{t-s} 
\left(v^N_{\lambda\phi}(s,\cdot)^{1 + \beta} \right) \dd s\rdual
\\
& \leq 
C \ldual X_0^N, \int_0^t S^N_{t-s} 
\left(S^N_s(\lambda\phi)^{1 + \beta} \right) \dd s\rdual
\\
& \leq 
C ||\phi||_\infty^{1+\beta} \lambda^{1 + \beta} t \ldual X_0^N,\ind\rdual 
.
\end{align*}
Substituting in~(\ref{theta moment bound}),
\begin{multline*}
\IE[ \ldual X^N_t, \phi \rdual^{1 + \theta} ]
\leq
2 + C \int_1^\infty y^{1 + \theta} \int_0^{\frac{2}{y}}
\Big[||\phi||_\infty^{1+\beta} \lambda^{1 + \beta} 
\Big( \ldual X_0^N,\ind\rdual^{1+\beta}+t\ldual X_0^N, \ind\rdual\Big) 
\\+\cO(\varepsilon^N)\|\lambda \phi\|_\infty^2t\ldual X_0^N,\ind\rdual
\Big]\dd \lambda \dd y
,
\end{multline*}
and since $\theta<\beta$ the right hand side is finite.

An application of the Monotone Convergence Theorem yields
\begin{align}
\label{bound to use in semimart decomp}
\IE[ \ldual X^N_t, \ind \rdual ^{1 +\beta} ] \leq H'
,
\end{align}
with $H'<\infty$ independent of $t\leq T$.

We have now established a bound of the form
\begin{align*}
\IE[ \ldual X^N_t, \phi \rdual^{1 + \theta} ]
\leq 
2 + C \int_1^\infty y^{1+\theta} \int_0^{\frac{2}{y}} \lambda^{1 +\beta} + 
\lambda^2 \dd \lambda 
\dd y \left( \ldual X^N_0, \ind \rdual^{1 +\beta} + \ldual X^N_0, \ind  \rdual \right)
\leq H''
.
\end{align*}
The proof now follows the pattern established in 
Lemmas~\ref{condn 1 of propn fixed radius case}
and~\ref{lemma for projections}.
Taking the terms in
the semimartingale decomposition of $\ldual X^N_t,\phi\rdual$ one by one,
first observe that by Jensen's inequality, 
$$\IE\left[\left(\frac{1}{t}\int_0^t \ldual X^N_s, |\cA^N(\phi)|\rdual\dd s
\right)^{1+\theta}\right]\leq \IE\left[\frac{1}{t}\int_0^t
\ldual X_s^N, |\cA^N\phi| \rdual^{1+\theta}\dd s\right],$$
which we can bound using~(\ref{bound to use in semimart decomp}) and
the fact that, from Lemma~\ref{uniform lipschitz},
$|A^N\phi|\leq m(N)\|\phi\|_{C^3}$. 
Rearranging the semimartingale decomposition to give an expression 
for $M_t^N(\phi)$ and
combining with Minkowski's inequality we can bound
$\IE[|M^N_T(\phi)|^{1+\theta}]$, and then by Doob's martingale
inequality
we see that
\begin{multline*}
\IE[ \sup_{s \leq T} \ldual X^N_s, \phi \rdual^{1 + \theta} ]
\\
 \leq
C\left(\IE[ \ldual X_0^N , \phi \rdual ^{1+\theta} ]
+
\IE\left[
\left(\int_0^T |\cL^N (\phi_s)(X_s^N)| ds \right)^{1 +\theta} \right]
 + \IE[ 
|M^N_T(\phi)^{1 + \theta}| ]
\right),
\end{multline*}
and, since we have established uniform bounds on each of the terms on 
the right hand side of this expression, this completes the proof.
\end{proof}

To conclude the proof of tightness of $\{X^N\}_{N\geq 1}$
we shall use a criterion due to Aldous. 
\begin{thm}[\cite{aldous:1978}]
A sequence of c\`adl\`ag real-valued processes $\{Y^N\}_{N\geq 1}$ 
is tight if and only if the following two conditions hold:
\begin{enumerate}
\item
for each fixed $t$, $\{Y^N_t\}_{N\geq 1}$ is tight in $\IR$;
\item
for any $\varepsilon>0$, 
given a sequence of stopping times $\tau_N$ bounded above by $T$ and 
a sequence of real numbers $\delta_N \to 0$ as $N \to \infty$, 
\begin{align*}
\lim_{N \to \infty} \IP[|Y^N_{\tau_N + \delta_N} - Y^N_{\tau_N} |>\varepsilon ] = 0
.
\end{align*}
\end{enumerate}
\end{thm}
The first condition is trivially satisfied, and so it remains to prove the
following lemma.
\begin{lemma}
Under the conditions of Theorem~\ref{jonoinfinitelimit}, for 
non-negative $\phi\in C_0^3(\IR^d)$,
given $\varepsilon>0$, a sequence 
of stopping times $\tau_N$ bounded above by $T$, and 
a sequence of real numbers $\delta_N \to 0$ as $N \to \infty$, 
\begin{align*}
\lim_{N \to \infty} \IP[|Y^N_{\tau_N + \delta_N} - Y^N_{\tau_N} |>\varepsilon ] = 0
.
\end{align*}
\end{lemma}
\begin{proof}
Again following the proof of Theorem~4.6.2 in \cite{dawson:1993},
we study the joint distribution of 
$\ldual X^N_{\tau^N+\delta^N},\phi\rdual$ and 
$\ldual X^N_{\tau^N}, \phi\rdual$ 
through the Laplace transform 
\begin{align*}
L^N(\delta_N; \lambda_1,\lambda_2) := \IE \left[ \exp\left(
-\lambda_1\ldual X^N_{\tau_N + \delta_N}, \phi \rdual - \lambda_2 \ldual X^N_{\tau_N} ,\phi \rdual
\right) \right]
.
\end{align*}
The approximate duality~(\ref{approximate duality relation}),
combined with the strong Markov property, gives
\begin{align*}
L^N(\delta_N; \lambda_1,\lambda_2) 
:= &
\IE \left[ \IE \left[ \exp\left(
-\lambda_1\ldual X^N_{\tau_N + \delta_N}, \phi \rdual - \lambda_2 \ldual X^N_{\tau_N} ,\phi \rdual
\right) 
\big| X^N_{\tau_N}
\right] \right]
\\
= &
\IE \left[ \exp\left(
-\ldual X^N_{\tau_N}, v^N_{\lambda_1\phi}(\delta_N,\cdot) \rdual - 
\ldual X^N_{\tau_N} ,v^N_{\lambda_2\phi}(0,\cdot) \rdual
\right) \right]
+ \cO(\varepsilon^N)\|\lambda_1\phi\|^2\ldual X_0^N,\ind\rdual
\\
=&
\IE \left[ \exp\left(
-\ldual X^N_{\tau_N}, v^N_{\lambda_1\phi}(\delta_N,\cdot) +
v^N_{\lambda_2\phi}(0,\cdot) \rdual
\right) \right]
+ \cO(\varepsilon^N)\|\lambda_1\phi\|^2\ldual X_0^N,\ind\rdual
,
\end{align*}
from which
\begin{multline*}
|L^N(\delta_N ; \lambda_1,\lambda_2) - L^N(0 ; \lambda_1,\lambda_2)|
\leq 
||v^N_{\lambda_1 \phi}(\delta_N,\cdot) - v^N_{\lambda_1\phi}(0,\cdot)||_\infty
\IE[ \sup_{s\leq T} \ldual X^N_s, \ind \rdual]
+ \cO(\varepsilon^N)\|\lambda_1\phi\|^2\ldual X_0^N,\ind\rdual
\\
\leq 
||v^N_{\lambda_1 \phi}(\delta_N,\cdot) - v^N_{\lambda_1 \phi}(0,\cdot)||_\infty
\IE[ 1 + \sup_{s \leq T} \ldual X^N_s, \ind \rdual^{1 +\beta}]
+ \cO(\varepsilon^N)\|\lambda_1\phi\|^2\ldual X_0^N,\ind\rdual
.
\end{multline*}
Now from~(\ref{vcharacterisationintermsofs})
it follows that
$\|v^N_{\lambda_1\phi}(\delta_N,\cdot) 
- v^N_{\lambda_1 \phi}(0,\cdot)\|_\infty \to 0$ as $N \to \infty$ and so 
combining with the above,
$$
|L^N(\delta_N ; \lambda_1,\lambda_2) - L^N(0 ; \lambda_1,\lambda_2)| \to 0.
$$ 
Tightness of $\{\ldual X^N_{\tau_N + \delta_N}, \phi \rdual ,  
\ldual X^N_{\tau_N} ,\phi \rdual \}$ follows from 
Lemma~\ref{supboundthetamomentlemma}.
Taking a convergent subsequence, since
$|L^N(\delta_N ; \lambda_1,\lambda_2) - L^N(0 ; \lambda_1,\lambda_2)| 
\to 0$, the limit of $\{\ldual X^N_{\tau_N + \delta_N}, \phi \rdual , 
\ldual X^N_{\tau_N} ,\phi \rdual \}$ must be of the 
form $\{ Z, Z\}$ and so we see that 
$\ldual X^N_{\tau_N + \delta_N}, \phi \rdual - 
\ldual X^N_{\tau_N} ,\phi \rdual \to 0$ in probability, which completes 
the proof.
\end{proof}

Tightness is now proved.

\section{Proofs of Lemmas~\ref{independence requirement fixed radius} 
and~\ref{independence requirement variable radius}} 
\label{indeplinsection}

In this section we exploit the duality of the SLFV with a process
of coalescing lineages to prove 
the two key estimates provided by
Lemmas~\ref{independence requirement fixed radius} 
and~\ref{independence requirement variable radius}.
In the fixed radius case, it is straightforward to 
make the heuristic argument of Section~\ref{heuristics} rigorous and 
analogous arguments can be found in, for example, \cite{etheridge/freeman/penington/straulino:2015}.
However, in the variable radius case, we must work a little harder
to recover the result claimed here. We present an approach that
works in either setting, but to avoid repetition restrict ourselves 
to the variable radius case.

\subsection{A coupling of ancestral lineages}

The key to our proof is the duality of the SLFV from 
Proposition~\ref{prop: dual}.
We first consider
\begin{equation}
\label{square integral}
\IE\left[
\int_{\IR^d}
\left(\frac{1}{|\cB_r^M|}\int_{\cB_r^M(x)}\phi(y) X_t^N(y)\dd y\right)^2
\dd x \right]
\end{equation}
for fixed $r$ and $t$. 

Using Proposition~\ref{prop: dual}, for integrable functions $\psi$,
for the scaled process,
\begin{equation*}
\IE\left[\int_{\IR^d} \int_{\IR^d} \psi(y,z) w^N_t(y)w^N_t(z) \dd y \dd z  \right]
= 
\int_{\IR^d} 
\int_{\IR^d} 
\psi(y,z)
\IE_{(y,z)}
\left[
\prod_{i=1}^{N_t} w_0^N(\xi^{N,i}_t)
\right]
\dd y \dd z,
\end{equation*}
where $\{\xi_t^{N,1},\xi_t^{N,2}\}$ are the positions of two ancestral 
lineages started from $y$ and $z$ respectively at time $0$. In the scaled
ancestral
process, events of radius $r/M$ covering the point $y$ fall at rate 
$Nr^d\mu^N(\dd r)$. During such an event, each lineage in the region covered
by the event, independently, is affected with probability $u(r)/J$.
Those lineages that are affected all jump to the same
new position which is
chosen uniformly at random from the affected region.
To estimate~(\ref{square integral}) we take
$$\psi(y,z)=\frac{1}{|\cB_r^M|^2}
\int_{\IR^d}\phi(y)\phi(z)\ind_{|y-x|<r/M}\ind_{|z-x|<r/M}\dd x$$
and recall that $X^N=Kw^N$.
The key tool is a coupling.
\begin{propn}
\label{coupling}
Let $\{\xi_t^{N,1}, \xi_t^{N,2}\}_{t\geq 0}$ be the scaled ancestral
lineages above. Then 
$$(\xi_t^{N,1},\xi_t^{N,2})\stackrel{d}{=}\big(Y_t^{N,1}+\varepsilon_t^{N,1},
(Y_t^{N,2}+\varepsilon_t^{N,2})\ind_{t<\tau}
+(Y_t^{N,1}+\varepsilon_t^{N,1})\ind_{t\geq \tau}\big),
$$
where $Y^{N,1}$ and $Y^{N,2}$ are independent random walks 
and the random coalescence time $\tau$ has hazard rate bounded above when 
$|\xi_t^{N,1}-\xi_t^{N,2}|=2r_0/M$ by
\begin{equation}
\label{hazard rate}
h(r_0)=C\frac{M^2}{J}\frac{1}{(r_0\vee J^{-1/\gamma})^{\gamma-\beta(\gamma-d)}},
\end{equation}
with $C$ a constant independent of $N$. Moreover, for $d\geq 2$ 
\begin{equation}
\label{bound on epsilon}
\IE[\sup_{0\leq t\leq T}M^2 J^{2/\gamma}
|\varepsilon_t^{N,1}-\varepsilon_t^{N,2}|^2]
\to 0\quad\mbox{as } N\to\infty.
\end{equation}
whereas if $d=1$, 
\begin{equation}
\label{bound on epsilon d=1}
\IE[\sup_{0\leq t\leq T}J 
|\varepsilon_t^{N,1}-\varepsilon_t^{N,2}|^2]
=\cO\big( 1 \big).
\end{equation}
\end{propn}
Before proving this result, let us see why it helps.
The error terms $\varepsilon_t^{N,i}$ are very far from being independent, but
they are very small. 
Notice in particular that if $\xi^{N,1}$ and $\xi^{N,2}$ have coalesced
by time $t$, then since $Kw^N_0=X_0^N$ converges weakly to $X_0$,
\begin{equation*}
\int_{\IR^d} 
\int_{\IR^d} 
\psi(y,z)
\IE_{(y,z)}[w_0^N(\xi^{N,1}_t)]
\dd y \dd z,
\end{equation*}
is of order $1/K$, which when multiplied by $K^2$ to give us the 
contribution to~(\ref{square integral}) is $\cO(K)$.
If, on the other hand, the two lineages have not coalesced, 
then they are close to independent random walks and
\begin{equation*}
\int_{\IR^d} 
\int_{\IR^d} 
\psi(y,z)
\IE_{(y,z)}[w_0^N(\xi^{N,1}_t)w_0^N(\xi^{N,2}_t)]
\dd y \dd z=\cO\big(\frac{1}{K^2}\big).
\end{equation*}

Under the assumptions of 
Theorem~\ref{result fixed radius}, $(KN)/(J^2M^d)\to C_2$ as
$N\to\infty$, and so the proof of 
Lemma~\ref{independence requirement fixed radius} 
is reduced to checking that the coalescence probability 
tends to zero as $N\to\infty$.
Under the 
conditions of Theorem~\ref{jonoinfinitelimit}, 
\begin{equation}
\label{integral of rates 1}
\frac{N}{J^2M^d}\int_0^\infty u(r)^2 r^{2d}\mu^N(dr)=
\cO\left(\frac{N}{J^2M^d}J^{(1-\beta)(\gamma-d)/\gamma}\right)
=\cO\left(\frac{M^{2-d}}{J}J^{(1-\beta)(\gamma-d)/\gamma}\right).
\end{equation}
In $d\geq 2$, since $(1-\beta)(\gamma-d)/\gamma<1$, we easily see that 
this tends to zero as $N\to\infty$. The sparsity 
condition~(\ref{infvarrw}) guarantees that the same is true in $d=1$ and
so the contribution to the expression in 
Lemma~\ref{independence requirement variable radius}
from lineages that have not coalesced is negligible. 
On the other hand, multiplying the expression in~(\ref{integral of rates 1})
by $K$ we obtain an expression of order
$$\cO\left(\frac{KN}{J^2M^d}J^{(1-\beta)(\gamma-d)/\gamma}\right)=
\cO\left(\Big(\frac{J}{N}\Big)^{(1-\beta)/\beta}J^{(1-\beta)(\gamma-d)/\gamma}\right),
$$
which grows without bound as $N\to\infty$ (because of
Condition~\ref{infvarcond3} of Theorem~\ref{jonoinfinitelimit}). To prove
Lemma~\ref{independence requirement variable radius}
we must show that this quantity multiplied by the 
coalescence probability tends to zero. In $d\geq 2$, under the 
conditions of Theorem~\ref{jonoinfinitelimit}, the coalescence 
probabilities will be of the same order as in the fixed radius case, and
so the Lemma will follow easily  
(again using $(1-\beta)(\gamma-d)/\gamma<1$). In one
dimension, things are more delicate, and we shall see the need for 
our more stringent sparsity condition.

~

\begin{proof}[Of Proposition~\ref{coupling}]

We rewrite the Poisson process of events $\Pi^N$ that drives the dynamics
of $\xi^{N,1}$ and $\xi^{N,2}$ as the sum of four components through 
a thinning. Each event $(x,t,r,\rho)\in\Pi$ is augmented by two 
independent Bernoulli random variables, $\eta_1$, $\eta_2$, each with
success probability $\rho=u(r)=r^{-\gamma}$. The random variable 
$\eta_i$ determines whether or not the $i$th ancestral lineage is 
affected by the event. We now let $\Pi^{N,1}$ be the 
events in $\Pi$ for which $\eta_1=1, \eta_2=0$; $\Pi^{N,2}$ is the subset of
events with $\eta_1=0, \eta_2=1$ and $\Pi^{N,1,2}$ is the events with
$\eta_1=\eta_2=1$. The remaining events won't affect the motion of either
lineage.

The lineage $\xi_t^{N,i}$ is driven by 
$\Pi^{N,i}\bigcup\Pi^{N,1,2}$ (recalling that after coalescence there is
a single lineage, $\xi^{N,1}$). Coalescence results from both lineages
being in the region covered by an event from $\Pi^{N,1,2}$.  
If $|\xi_t^{N,1}-\xi_t^{N,2}|=2r_0/M$, then the rate at which they are both
in the region affected by an event from $\Pi^{N,1,2}$ is bounded above by
$$C'\frac{N}{J^2}\int_{r_0}^1r^du(r)^2\mu^N(\dd r)=
C'\frac{N}{J^2}\int_{r_0\vee J^{-1/\gamma}}^1 r^{d-2\gamma+\alpha}\dd r
\leq C\frac{M^2}{J}\frac{1}{(r_0\vee J^{-1/\gamma})^{\gamma-\beta(\gamma-d)}},
$$
where we have approximated the volume of the centres $x$ such that
$\cB_r(x)$ contains two points at separation $2r_0$ by $C(d)r^d$ for $r>r_0$
and we have used that $\alpha+1=(\beta +1)(\gamma -d)$ and
$N/(JM^2)\to C_1$ as $N\to\infty$.

Let $Y^{N,i}$ be the random walk driven by the events of $\Pi^{N,i}$ and 
$\varepsilon_t^{N,i}=\xi_t^{N,i}-Y_t^{N,i}$. Evidently, since they are driven by
independent Poisson processes of events, the walks $Y^{N,1}$ and $Y^{N,2}$ are
independent. It remains to bound $|\varepsilon^{N,1}-\varepsilon^{N,2}|$.
First observe that we can couple our lineages $\xi_t^{N,i}$ with a system
$\widetilde{\xi}_t^{N,i}$ of lineages that do not coalesce, by at each
event of $\Pi^{N,1,2}$ choosing {\em two} parental positions
independently and uniformly at random from the affected region
and, if both lineages are in the affected region, the 
lineage $\widetilde{\xi}_t^{N,i}$
jumps to the $i$th parental location. 
Until the coalescence time $\tau$, the processes $\widetilde{\xi}_t^{N,i}$
and $\xi_t^{N,i}$ coincide. We define $\widetilde{\varepsilon}_t^{N,i}$
in the obvious way. Evidently, it will suffice to control the supremum of 
$|\widetilde{\varepsilon}^{N,1}-\widetilde{\varepsilon}^{N,2}|$ over the
time interval $[0,T]$.
 
First observe that
$$M_{\widetilde{\varepsilon}}(t):=\widetilde{\varepsilon}_t^{N,1}-
\widetilde{\varepsilon}_t^{N,2}$$ 
is a mean
zero martingale.
We should like
to estimate its variance at time $t$. It is driven entirely by events
from $\Pi^{N,1,2}$ and by 
exploiting translation invariance of the system, we find that
\begin{eqnarray*}
\IE[\ldual M_{\widetilde{\varepsilon}}\rdual_t]&\leq & C
\int_0^t\frac{N}{J^2}\int_{1/J^{1/\gamma}}^1\big(\frac{r}{M}\big)^2
r^{-2\gamma+d+\alpha}\dd r\dd s\\
&=& \left\{\begin{array}{ll}
\cO\big(\frac{1}{J}\big)& \mbox{if }\beta(\gamma-d)-\gamma+2>0;\\ \\
\cO\big(\frac{\log J}{J}\big)& \mbox{if }\beta(\gamma-d)-\gamma+2=0;\\ \\
\cO\big(\frac{1}{J^{(2+\beta(\gamma-d))/\gamma}}\big)& \mbox{if }
\beta(\gamma-d)-\gamma+2<0. 
\end{array}\right.
\end{eqnarray*}
Since $\beta(\gamma-d)-\gamma+2=2-d-(1-\beta)(\gamma-d)$ and under
our assumptions $(1-\beta)(\gamma-d)<1$, if $d=1$ the first line
holds, whereas 
if $d\geq 2$ the third line holds.
In particular, for $d\geq 2$, using Condition~\ref{infvarcond3}
of Theorem~\ref{jonoinfinitelimit}, 
the quantity on the right is always 
$o(1/(M^2J^{2/\gamma}))$ from which~(\ref{bound on epsilon}) easily follows via
Doob's maximal inequality. An application of Doob's inequality also gives the
result in $d=1$.
\end{proof}

\subsection{Coalescence Probabilities}

It remains to estimate our coalescence probabilities.
\begin{propn}
\label{coalescence probabilities}
Let the scaled ancestral lineages $\xi^{N,1}$, $\xi^{N,2}$ be as in 
the previous subsection. We start them from two points sampled 
independently at random from $\cB^M_r(0)$. The probability that they
have coalesced by time $T$, which we shall denote by $p_N(T)$ satisfies
\begin{equation}
\label{hazard bounds}
p_N(T)\leq\left\{
\begin{array}{ll}
C J^{(1-\beta)(\gamma-1)/\gamma} \; \frac{M^2}{J}
&\mbox{in }d=1;\\
\\
C \frac{\log M}{J} &\mbox{in }d=2;\\
\\
C\frac{1}{J}&\mbox{in }d\geq 3.
\end{array}\right.
\end{equation}
\end{propn}
In the light of the discussion immediately after the statement of 
Proposition~\ref{coupling}, this result will complete 
the proof of our key lemma in the 
variable radius case. 

\begin{proof}[of Proposition~\ref{coalescence probabilities}]
Our aim is to show that 
\begin{equation}
\label{total hazard bound}
\IE\Big[\int_0^T h(M|\xi_s^{N,1}-\xi_s^{N,2}|)\dd s\Big]
\end{equation}
is bounded by the quantities on the right hand side of~(\ref{hazard bounds}).
where $h(r)$ is given by~(\ref{hazard rate}). 
To do this, we shall 
use the coupling of Proposition~\ref{coupling}
to approximate $|\xi^{N,1}-\xi^{N,2}|$ 
by $|Y^{N,1}-Y^{N,2}|$ 
and then the classical Fourier transform approach of \cite{darling/kac:1957} to estimate the corresponding
expected integral of the hazard rate.
Since we are only interested in the separation of our lineages we denote $Y^N:= |Y^{N,1} - Y^{N,2}|$ and $\varepsilon^N:= |\varepsilon^{N,1} - \varepsilon^{N,2}|$.

We first observe that
\begin{multline}\label{hazardratedecomposition}
\IE\Big[\int_0^{T}h\big(M(Y^N_s+\varepsilon_s^N)\big)\dd s\Big]
\\
=\int_0^T \IE[h(MY^N_s)+M\varepsilon_s^Nh'(MY_s^N)+(M\varepsilon_s^N)^2
h''\big(MY_s^N+\theta(s)M\varepsilon_s^N\big)]ds
,
\end{multline}
for some $\theta(s)\in(0,1)$ and that
\begin{equation} \label{hazardratederivatives}
h'(r)=Ch(r)\frac{1}{r}\ind_{[J^{-1/\gamma},1]}(r)\leq J^{1/\gamma}h(r);
\quad
h''(r)=C'h(r)\frac{1}{r^2}\ind_{[J^{-1/\gamma},1]}(r)\leq J^{2/\gamma}h(r)
.
\end{equation}
We now begin with $d\geq 2$.
Using~(\ref{bound on epsilon}), we see that
$$\IE\Big[\int_0^{T}h(M(Y^N_s+\varepsilon_s^N)\dd s\Big]
\leq C \int_0^T \IE[h(MY^N_s)]ds.$$
The random walk $Y^N$ jumps at exponentially distributed times with mean 
$\cO(1/M^2)$, and so
will make $\cO(M^2)$ such jumps by time $T$. 
It is evidently enough to show that 
$$\int_0^{T_N} \IE[h(MY^N_s)]ds$$
is bounded by the quantities on the right of~(\ref{hazard bounds})
where $T_N$ is the random time at which the walk $Y^N$ has made 
a geometric number of jumps with mean $M^2$.
Let us write $p(x,y)$ for the probability density function of 
the unscaled displacement $MY^N$ at a single jump and
$p^{\circ k}(x,y)$ for its $k$-fold convolution.
Further write $\cT_M$ for the time of the $G^M$th jump of $Y^1$, where
$G^M$ is an independent geometrically distributed random variable with
mean $M^2$.
Then 
\begin{multline}
\label{IM}
\IE\Big[\int_0^{T_N} h(MY^N_s)\dd s\Big]
=\frac{1}{M^2}\IE\Big[\int_0^{\cT_M}h(Y_s^1)\dd s\Big]
\\
\leq \frac{1}{M^2}\sum_{k=0}^{\infty}\int_{\IR^d} z_M^k p^{\circ k}(0,y)h(|y|)dy
:= Q_M
\end{multline}
where $z_M=1-1/M^2$ and on the right hand side we have bounded the 
expectation above by setting $Y_0=0$.
Define $I_M(x):= \sum_{k=0}^{\infty}z_M^k p^{\circ k}(0,x)h(|x|)$. Writing $\phi(t)$ for the characteristic function of the displacement of
$Y$ in a single jump and taking Fourier transforms, 
$$\widehat{I}_M(t)= \Big(\frac{1}{1-z_M\phi}\circ \widehat{h}\Big)(t).$$
Since $Y$ is rotationally symmetric with strictly bounded jumps, its 
characteristic function is real-valued and $\sim 1-C|t|^2$ close to 
the origin and 
$$\frac{1}{1-z_M\phi(t)}\leq \frac{C'}{\frac{1}{M^2}+C|t|^2}.$$

In $d\geq 3$ we may let $M\to\infty$ in this 
expression (corresponding to an infinite time
horizon) and using
that the inverse Fourier transform of $|t|^{-2}$ is $|x|^{2-d}$ and, recalling the definition of $h$ from \eqref{hazard rate}, we can approximate $Q_M$ by
$$
\frac{1}{M^2} \int_{\IR^d}I_M(y) dy \leq \frac{C}{M^2}\frac{M^2}{J}\int_0^1 r^{2-d}r^{d-1}
\frac{1}{r^{\gamma-\beta(\gamma-d)}}\dd r
=\frac{C}{J}\int_0^1r^{1-\gamma + \beta(\gamma-d)}\dd r
\propto\frac{1}{J},$$
where we obtained an upper bound by replacing the lower limit
of integration by zero.

In $d=2$, 
the inverse Fourier transform of $1/(\epsilon^2+|t|^2)$ is 
$K_0(\epsilon |x|)$, where $K_0$ is the modified Bessel function 
of the second kind of order zero (Eq.~17, p.365, \cite{gelfand/shilov:1964})
.
Recalling that $K_0(z)\sim\log(1/z)$ as $z\to 0$, 
and substituting in~(\ref{IM}) yields 
$Q_M \leq C(\log M)/J$.

We now consider $d=1$.
Looking again at $Q_M$, we can no longer replace the lower limit in the
integral by zero as
$r^{-2\gamma+\alpha+1}$ is not integrable at the origin. Instead 
observe that (by a change of variables)
the inverse Fourier transform of $(1/M^2+|t|^2)^{-1}$ 
at $x$ is 
$M$ times that of $1/(1+|t|^2)$ evaluated at $x/M$, which,
using Eq.40, p.363 of \cite{gelfand/shilov:1964} and writing $K_{1/2}$ for the modified Bessel function 
of the second kind of order $1/2$,
$$\sqrt{\frac{|x|}{M}}K_{1/2}\left(\frac{|x|}{M}\right)\sim 1
\quad\mbox{as }M\to \infty.$$
Substituting,
\begin{eqnarray*}
Q_M \leq \frac{CM^2}{J}\int_0^1\frac{1}{M}\int^1_{r_0\vee 1/J^{1/\gamma}}r^{-2\gamma+\alpha+1}\dd r \dd r_0
&=&C\frac{M}{J}\frac{1}{J^{1/\gamma}}\Big(\frac{1}{J^{1/\gamma}}\Big)^{-2\gamma + \alpha +2}
+C'\frac{M}{J}\\
&=&C''\frac{M}{J}J^{(1-\beta)(\gamma-1)/\gamma}.
\end{eqnarray*}
In other words the main contribution comes from separations at most 
$1/J^{1/\gamma}$.
This bounds the contribution from $h(MY^N_s)$ in \eqref{hazardratedecomposition}. From~(\ref{bound on epsilon d=1}), $M|\widetilde{\varepsilon}^N|$ is at most
$\cO(M/\sqrt{J})$ and so
if $M/\sqrt{J}\leq \cO(J^{-1/\gamma})$,
by using \eqref{hazardratederivatives} as in two dimensions, we obtain a bound which is $1/M$ times
the expression on the right hand side of~(\ref{hazard bounds}) and we
are done. If not then the order
of the error is $MJ^{1/\gamma}/\sqrt{J}$ times the quantity above and, 
since by assumption $\gamma>2$, this completes the proof.
\end{proof}

\section{Showing convergence}
\label{convergence section}

Armed with tightness, it remains to identify the limit points of the
sequences $\{X^N\}_{N\geq 1}$.

\subsection{Proof of Theorem~\ref{result fixed radius}}

Recall that for each $N$, using Lemma~\ref{roughgenerator},
$$M^N(\phi)_t=\ldual X_t^N,\phi\rdual-\ldual X_0^N,\phi\rdual
-\int_0^t \ldual X_s^N,\cA^N(\phi)\rdual ds -\zeta^N_t(\phi)$$
is a martingale.

From the convergence of $m(N)$,
\begin{align*}
\lim_{N \to \infty} \IE 
\left[
\sup_{t \leq T}
\left\abs
\int_0^t \ldual X^{N}_s, \cA^{N}(\phi) \rdual \dd s + \zeta^N_t(\phi)
-
\int_0^t \ldual X_s, \frac{m}{2} \Delta \phi \rdual \dd s
\right\abs
\right]
= 0
,
\end{align*}
where $m$ is as in the statement of
Theorem~\ref{result fixed radius}.
Moreover, using~(\ref{quadratic variation})
and Lemma~\ref{independence requirement fixed radius} 
$$\lquad M^N(\phi)\rquad_t\to 2\kappa\int_0^t\ldual X_s^N,\phi\rdual \dd s$$
as $N\to\infty$.

Suppose that $\{X^{N_k}\}_{k\geq 1}$ is a convergent subsequence.
To prove that the limit is the finite variance superBrownian motion, we
must check that the martingale property is preserved under passage to
the limit and that for each 
$t\leq T$, $M_t^{N_k}(\phi)^2+\lquad M^{N_k}(\phi)\rquad_t$
is uniformly integrable (see e.g.~Lemma~II.4.5 in \cite{perkins:2002}).
To check the first condition, suppose that 
$0 \leq t_1 <...< t_k \leq s < t$, $h_1,...,h_k \in C_b$ and 
$\phi \in C_0^3(\IR^d)$. An application of the Dominated Convergence Theorem
yields
\begin{align*}
\IE \left[ 
\bigg(
\ldual X_t , \phi \rdual  - \ldual X_s, \phi \rdual
- \int_s^t \ldual X_u , A \Delta \phi \rdual \dd u
\bigg)
\prod_{j= 1}^k h_j( \ldual X_{t_j}, \phi \rdual)
\right]
= 0
.
\end{align*}
Thus, see e.g.~Theorem~4.8.10~of \cite{ethier/kurtz:1986},
\begin{align*}
M_t(\phi) := \ldual X_t, \phi \rdual  - 
\ldual X_0, \phi \rdual - \int_0^t \ldual X_s, \frac{m}{2} \Delta \phi \rdual \dd s
,
\end{align*}
is a martingale.

The uniform integrability of $\lquad M^N(\phi)\rquad_t$ is immediate from
our previous calculations. That of $M^N_t(\phi)^2$ follows on setting
$\theta=1$ in the  
arguments (based on Jensen and Minkowski's inequalities) that gave us
Lemma~\ref{supboundthetamomentlemma}. 

Since the solution to the martingale problem corresponding to the finite
vairance superBrownian motion is unique, the proof of 
Theorem~\ref{result fixed radius} is complete.

\subsection{Proof of Theorem~\ref{jonoinfinitelimit}}

Since the solution to the martingale problem for the superBrownian motion
with stable branching law is unique, it is enough to check that all 
limit points satisfy the martingale problem~(\ref{beta branching martingale}).
Once again we must check that the martingale property is conserved 
under passage to the limit, for which
it suffices to
show that for any $k \geq 0$, $0 \leq t_1 <...< t_k \leq s < t$, 
$h_1,...,h_k \in C_b$ and $\phi \in C_0^3(\IR^d)$ with $\phi \geq 0$,
\begin{multline*}
\lim_{N\to\infty}
\IE
\Bigg[
\bigg(
\exp(-\ldual X^N_{t}, \phi \rdual)
-
\exp(-\ldual X^N_s, \phi \rdual)
\\
-
\int_s^t \left\ldual X^N_u , -\frac{m}{2}\Delta \phi + 
\kappa \phi^{1+\beta} \right\rdual \exp(-\ldual X^N_u,\phi \rdual)
\dd u
\bigg)
\prod_{j= 1}^k h_j( \ldual X_{t_j}, \phi \rdual)
\Bigg]
= 0
,
\end{multline*}
with $m$ and $\kappa$ as in the statement of Theorem~\ref{jonoinfinitelimit}.
Once again, since the $h_j$ are bounded, this is an easy consequence of the
Dominated Convergence Theorem, our results of 
Section~\ref{variableradiuslimiting}
and Lemma~\ref{independence requirement variable radius}.

\textbf{Acknowledgments.}
AE would like to thank Tom Kurtz for useful discussions.

\bibliographystyle{plainnat}
\DeclareRobustCommand{\VAN}[3]{#3}

\bibliography{structureCDE}

\end{document}